\documentclass[final,leqno]{siamltex}
\pdfoutput=1
\usepackage{amsmath, amsfonts}
\usepackage{graphicx, xcolor, caption}
\usepackage{subcaption}
\usepackage{multirow,enumerate}
\usepackage{tikz}
\usepackage{booktabs}
\usepackage{multirow}
\usepackage{algorithm, algpseudocode}
\usepackage{tikz-qtree}
\usepackage{mathrsfs}
\tikzset{every tree node/.style={minimum width=1em,draw,circle, inner sep=1pt},
         blank/.style={draw=none},
         edge from parent/.style=
         {draw, edge from parent path={(\tikzparentnode) -- (\tikzchildnode)}},
         level distance=1.75cm,
         sibling distance=0.2cm}
\usetikzlibrary{patterns}

\numberwithin{algorithm}{section}

\newcommand{\BMAT}{\begin{bmatrix}}
\newcommand{\EMAT}{\end{bmatrix}}

\newcommand{\ie}{\textit{i.e.}}
\newcommand{\eg}{\textit{e.g.}}

\newcommand{\rskelf}{\texttt{rskelf}}
\newcommand{\parent}{\text{parent}}
\renewcommand{\P}{\mathscr{P}}
\newcommand{\U}{\mathscr{U}}
\newcommand{\new}[1]{\bar #1}

\newcommand{\hif}{\texttt{hif}}

\newcommand{\child}{\text{child}}

\newcommand{\rd}{\mathcal{R}}
\newcommand{\M}{\mathscr{M}}
\newcommand{\sk}{\mathcal{S}}
\newcommand{\nbor}[1]{\mathcal{N}(#1)}

\renewcommand{\L}{\mathscr{L}}

\newcommand{\edge}{\text{edge}}
\newcommand{\voro}{\mathcal{V}}

\newcommand{\skel}{Z}

\newcommand{\J}{\mathcal{I}}
\newcommand{\F}{\mathcal{F}}
\newcommand{\N}{\mathcal{N}}
\renewcommand{\nbor}{\text{nbor}}

\usepackage{mathtools,mathabx}

\colorlet{lightgray}{black!20}
\colorlet{darkgray}{black!50}

\title{A technique for updating hierarchical skeletonization-based factorizations\\ of integral operators}

\author{
	Victor Minden\thanks{Institute for Computational and Mathematical Engineering, Stanford University, Stanford, CA 94305. Email: \texttt{\{vminden, damle\}@stanford.edu}}\and
	Anil Damle$^*$\and
	Kenneth L. Ho\thanks{Department of Mathematics, Stanford University, Stanford, CA 94305. Email: \texttt{klho@stanford.edu}}\and
	Lexing Ying\thanks{Department of Mathematics and Institute for Computational and Mathematical Engineering, Stanford University, Stanford, CA 94305. Email: \texttt{lexing@math.stanford.edu}}
}

\begin{document}

\maketitle

\begin{abstract}
  We present a method for updating certain hierarchical factorizations
  for solving linear integral equations with elliptic kernels.  In
  particular, given a factorization corresponding to some initial
  geometry or material parameters, we can locally perturb the geometry
  or coefficients and update the initial factorization to reflect this
  change with asymptotic complexity that is poly-logarithmic in the
  total number of unknowns and linear in the number of perturbed
  unknowns.  We apply our method to the recursive skeletonization
  factorization and hierarchical interpolative factorization and
  demonstrate scaling results for a number of different 2D problem
  setups.
\end{abstract}



\section{Introduction}
\label{sec:intro}
In engineering and the physical sciences, many fundamental problems of
interest can be expressed as an integral equation (IE) of the form
 \begin{align}\label{eq:bie}
  a(x) u(x) + b(x) \int_{\Omega} K(x,y) c(y) u(y) \, dy = f(x), \quad x \in \Omega \subset \mathbb{R}^{d},
  \end{align}
where $a(x),$ $b(x)$, and $c(x)$ are given functions typically
representing material parameters, $u(x)$ is the unknown function to be
determined, $K(x,y)$ is some integral kernel, $f(x)$ is some known
right-hand side, and the dimension $d=2$ or $3$.  Typically, $K(x,y)$
is associated with some underlying elliptic partial differential equation (\ie, it is the
Green's function or its derivative) and it thus tends to be singular
at $x=y$.

Discretizing the integral operator in \eqref{eq:bie} with $N$ degrees
of freedom (DOFs) via, \eg, the collocation, Nystr\"om, or Galerkin method reduces our problem to solving a linear system,
 \begin{align}\label{eq:linear}
Gu = f\,,
\end{align}
where the matrix $G\in\mathbb{C}^{N\times N}$ is dense and
$u\in\mathbb{C}^N$ and $f\in\mathbb{C}^N$ here are to be interpreted
as discretized versions of $u(x)$ and $f(x)$ in \eqref{eq:bie}. For a concrete example, in the case of simple piecewise-constant collocation with $\{x_j\}$ as the set of collocation points the discretization of \eqref{eq:bie} becomes
\begin{align}\label{eq:discbie}
a(x_i)u_i + b(x_i)\sum_j K_{ij}c(x_j)u_j= f(x_j),
\end{align}
which we solve for $u_i\approx u(x_i)$ with
\begin{align}
K_{ij} \approx \int_{\Omega_j}K(x_i, y)\,dy
\end{align}
where $\Omega_j$ is the local collocation subdomain of $x_j$.  In this case, the matrix $G$ in \eqref{eq:linear} has entries $G_{ij} = a(x_i)\delta_{ij} + b(x_i)K_{ij}c(x_j)$ with $\delta_{ij}$ the Kronecker delta and we see that the off-diagonal structure of $G$ is essentially dictated by the discretized kernel $K_{ij}$.

Given disjoint sets of unknowns $\J$ and $\J'$ corresponding to point sets $\{x_j\}_{j\in\J}$ and $\{x_j\}_{j\in\J'}$ that are physically separated, we assume that the corresponding off-diagonal subblocks $G(\J,\J')$ and $G(\J',\J)$ are numerically low-rank.  For example, this is well-known to be the case for elliptic kernels where $K(x,y)$ is smooth away from $x=y$.  This observation is the cornerstone of a number of fast (linear or quasi-linear time
complexity) direct algorithms for factoring $G$ and solving \eqref{eq:linear} using hierarchical spatial subdivision to
expose and take advantage of the inherent physical structure of the
underlying problem.

\subsection{Problem statement}\label{sec:localized}
In this paper, we consider a sequence of problems of the form
\eqref{eq:bie} that are related through \emph{localized perturbations}.  By a localized perturbation we mean that, given a matrix $G$ discretizing the original problem and a matrix $\new{G}$ discretizing the new problem, there is a small local subdomain $\tilde \Omega \subset \Omega$ such that for all index sets $\J$ and $\J'$ with corresponding points all \emph{not} in the modified subdomain $\tilde \Omega$, we have
\begin{align}
\new{G}(\J,\J') = G(\J,\J').
\end{align}
Put simply, blocks of the system matrix that correspond to degrees of freedom away from the modifications are unchanged.  Such local perturbations include (but are not limited to):
\begin{itemize}
\item Localized geometric perturbations (see Figure \ref{fig:rndsqr}),
  wherein the domain of integration $\Omega$ is modified and therefore a subset of discretization points of $\Omega$ may move or discretization points may be added or removed.
\item Localized coefficient perturbations, wherein the material
  parameters $a(x)$, $b(x),$ or $c(x)$ are modified in a local region.
\end{itemize}

\begin{figure}
\centering
\includegraphics[scale=0.325]{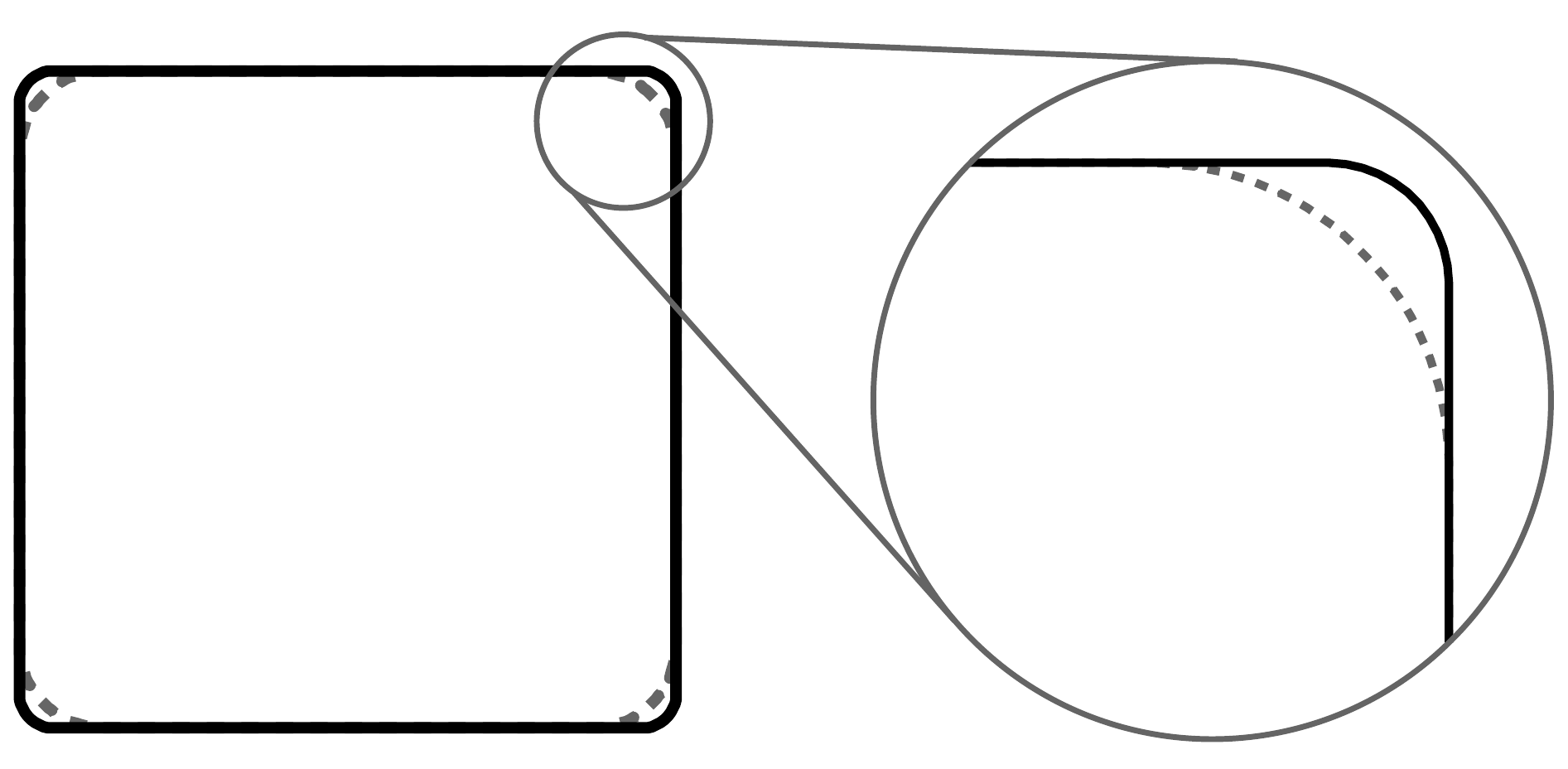}
\caption{As an example of a localized perturbation to the geometry, we start with the quasi-1D domain $\Omega=\Gamma_1$, the square with rounded corners following the dashed curve.  Then, for updating we adjust the rounding parameter to obtain $\Omega'=\Gamma_2$, the square with the sharper, solid corners. \label{fig:rndsqr} \vspace*{-.6cm}}
\end{figure}

By a sequence of localized updates, we mean that we are interested in applications where there are a number of localized perturbations
\begin{align}\label{eq:seq}
G=G^{(1)}\rightarrow G^{(2)} \rightarrow \dots \rightarrow G^{(i-1)} \rightarrow G^{(i)}\rightarrow\dots,
\end{align}
where each perturbation $G^{(i-1)}\to G^{(i)}$ is localized to some subdomain $\tilde \Omega_{i}$ that we allow to be different for each $i$.  Such sequences of problems can arise, \eg, in the
case of design problems where the physical system described by the
linear operator is a device that we want to design in an effort to optimize some
objective function.  We make the following observations:
\begin{itemize}
\item Localized perturbations lead to a global low-rank modification in the sense that entire rows and columns of the new matrix $G^{(i)}$ are different from the corresponding rows and columns in $G^{(i-1)}$, if such a correspondance even exists.
\item Because each perturbation can be localized to a different subdomain, for large $i$ the matrix $G^{(i)}$ is not necessarily given by a low-rank modification to $G$.
\end{itemize}

Because the perturbations we consider take advantage of the same physical structure used in the construction of hierarchical factorizations (\ie, spatial locality), it is not unreasonable to believe it might be possible to take a hierarchical factorization of $G^{(i-1)}$ and update it to obtain a hierarchical factorization of $G^{(i)}$.  This is what the method we describe in this work accomplishes in an efficient way for certain factorizations.

\subsection{Background}
Fast direct solvers for solving the linear systems arising from
discretized integral equations via the compression of low-rank blocks
exist in a number of different forms.  The seminal work in such
compressed representations is the $\mathcal{H}$- and
$\mathcal{H}^2$-matrices of Hackbusch \emph{et al.}
\cite{Hackbusch,HackbuschK,HackbuschB}, which provide an important
theoretical framework but in practice exhibit large constant factors
in the asymptotic scaling.

A hierarchical compression framework designed more explicitly to solve discretized elliptic integral
equations date back at least to \cite{martinsson-rokhlin} based on observations
in \cite{starr_rokhlin} and \cite{greengard-rokhlin}, and has since
been utilized and refined by a number of different authors (see, \eg,
\cite{gg,domainsAd,rskel,hifie}).  We refer to methods using this framework as
``skeletonization-based'' since at their core they employ the
interpolative decomposition for compression using the skeletonization
process described in \cite{id}.  Conceptually, these methods are
closely related to methods for systems involving so-called
hierarchically semi-separable (HSS) matrices (see, \eg,
\cite{fasthss,fastulv,xia}), and recent work has explicitly combined
the HSS and skeletonization frameworks \cite{corona2013}.  Notable
related schemes employing similar ideas include
\cite{chen,siva,bremer}.

The idea of updating matrix factorizations to solve sequences of
related systems is not a new one.  For example, in the linear
programming community it is common practice to maintain an LU
factorization of a sparse matrix $A$ that permits the addition or
deletion of rows/columns of $A$, or a general rank-one update, see \cite{gill}.  Further, it is well-known how to update the QR
factorization of a matrix after any of those same operations, see \cite{golub}.

The updating techniques described above, however, do not apply to fast
hierarchical factorizations.  Updating factorizations in the
$\mathcal{H}-$matrix format in response to local modifications has
been previously studied in \cite{djokic}, wherein a similar process to
this work is used to update the representation of the forward
operator, which allows for a post-processing step to obtain the
updated inverse in the same format.  Updating of the
skeletonization-based formats we consider here has not appeared thus
far in the literature, and, as we show, these formats admit efficient
one-pass updating.

In the case where the number of unknowns does not change and $\tilde \Omega_k$ is the same for all $k$, it is possible to order the unknowns in an LU decomposition such that those that will be modified are eliminated
last as in \cite{quintana}, which can be used to update LU factorizations for IE design problems where
only one small portion of the geometry is to be changed across all
updates.  Similarly, if the total number of unknowns modified between $G$ and $G^{(i)}$ is small and one is interested only in solving systems and not in updating factorizations, then for \emph{any} factorization of the base system $G$ it is relatively efficient to keep track of the updates as a global rank $k$ update $G^{(i)} = G + UCV$ with $U\in\mathbb{C}^{N\times k}$,
$C\in\mathbb{C}^{k\times k},$ and $V\in\mathbb{C}^{k \times N}$ and use the Sherman-Morrison-Woodbury (SMW) formula,
\begin{align}\label{eq:smw}
(G+UCV)^{-1} = G^{-1} - G^{-1}U\left(C^{-1}+VG^{-1}U\right)^{-1}VG^{-1},
\end{align}
taking advantage of the initial factorization of $G$ as is done in \cite{gg}.

\subsection{Contribution}
In this work we present a method to efficiently update skeletonization-based hierarchical factorizations in response to localized perturbations, \ie, to take a factorization corresponding to $G^{(i-1)}$ in \eqref{eq:seq} and obtain a factorization of $G^{(i)}$.  We illustrate our approach using the language of the recursive skeletonization factorization of \cite{hifde,rskel} and hierarchical interpolative factorization of \cite{hifie}, though our approach is simple to generalize to any factorization using bottom-up hierarchical compression of off-diagonal subblocks.

There are a number of advantages to our approach over using the SMW formula to solve a system with $G^{(i)}$.  In the case where the number of unknowns that have been modified between $G$ and $G^{(i)}$ is bounded by a small constant $m$ and the cost of solving a system with the existing factorization of $G$ is $\mathcal{O}(N)$, the cost of a solve using \eqref{eq:smw} (dropping terms that don't depend on $N$) is $\mathcal{O}(N + mN)$, where the second term can be amortized across multiple right-hand-sides.  However, if the number of total modified unknowns $m$ comprises any substantial fraction of $N$ then this is
not a viable strategy.

In contrast, under certain assumptions on the attainable compression of off-diagonal blocks in the factorizations considered in this paper, if the number of modified unknowns between two factorizations is bounded by $m$ then the asymptotic cost of our updating method is $\mathcal{O}(m\log^p N)$ for some small $p$.  Furthermore, one obtains a factorization of the new matrix and not just a method for solving systems.  This factorization can of course be subsequently efficiently updated, but can be useful for other reasons such as computing determinants or applying or solving with a matrix square root.

\section{Preliminaries}
\label{sec:review}

The updating ideas presented in this paper apply, in principle, to
many of the existing fast hierarchical algorithms for IEs.  For concreteness, we
present them in the context of quadtree-based generalized triangular
factorizations as presented in \cite{hifie}, in contrast to the
telescoping decompositions previously discussed
in, \eg, \cite{martinsson-rokhlin,rskel,domainsAd}. We begin by reviewing the linear algebra necessary for these factorizations to establish notation and elucidate the components of such factorizations that lead to efficient updating, though we direct the reader to \cite{hifie} for further details. For brevity, we
restrict our discussion to solving quasi-1D problems (\ie, curves in
the plane) and true 2D problems such that $\Omega \subset \mathbb{R}^2$, though the same basic process works
in 3D.  Further, our definitions and examples are given assuming collocation in which case degrees of freedom (DOFs) correspond to zero-dimensional point sets.  In the case of, for example, Galerkin discretization, where elements have nonzero spatial extent, certain definitions will need to be extended appropriately.

Recall that, given a set of DOFs corresponding to a discretization of \eqref{eq:bie}, construction of a hierarchical factorization of $G$ in \eqref{eq:linear} requires a way to expose compressible
interactions between sets of DOFs. In this work, we use a quadtree with $L$ levels, which we assume is constructed such that
leaf-level boxes each contain a number of DOFs bounded by an
occupancy parameter $n_{\text{occ}}$ independent of $N$.  In other
words, the tree is adaptive.  We note that this assumption implies
that construction of the hierarchical decomposition is a super-linear
process with complexity $\mathcal{O}(N\log N),$ but in practice
constructing the quadtree does not significantly contribute to
runtime.

In the remainder of this paper, we adopt the following notation.
For a positive integer $n$, we use $[n]$ to denote the index set
$\{1,2,\ldots,n\}$. Given a matrix $A\in\mathbb{C}^{n\times n}$, we
will use $\J\subset [n]$ and $\J' \subset [n]$ to denote disjoint sets
of DOFs, which will later on be explicitly associated with boxes in our quadtree.  For a given DOF set $\J$, we will write the complement DOF set as $\J^c = [n]\setminus \J.$  We use the MATLAB-style notation $A(\J,\J)$ to denote the diagonal subblock of $A$
corresponding to self-interactions between DOFs in $\J$ and $A(\J,\J')$ or $A(\J',\J)$ to denote off-diagonal subblocks of $A$ corresponding
to cross-interactions between the DOFs associated with the index sets
$\J$ and $\J'$.  This is in contrast to the simple subscript $A_\J$, which will be used to label a matrix $A$ that is in some way associated with $\J$.  We will use $A(:,\J)$ to refer to $A([n],\J)$ and define $A(\J,:)$ similarly.  In general, we will use upper-case variable names (\eg, $A$) to refer
to matrices and matrix-valued functions and variable names in
math-calligraphic font (\eg, $\mathcal{N}$) to refer to index sets or index-set-valued
functions (with the notable exception that $\mathcal{O}$ will be used for ``big-O'' notation), and variable names in math-script font (\eg, $\L$) to refer to collections of index sets.  When referring to specific boxes in the quadtree, we will always use the letter $b$, and similarly we will denote edges as $e$.

\subsection{Interpolative decomposition}\label{sec:id}
As discussed in Section \ref{sec:intro}, off-diagonal subblocks of $G$ corresponding to sets of DOF sets discretizing non-overlapping subdomains are assumed to be numerically low-rank, and thus fast algorithms for solving \eqref{eq:linear}
typically use some form of compression to approximate these subblocks.
One such method is the \emph{interpolative decomposition} (ID), which
we define below in a slightly non-standard fashion.

\begin{definition}\label{def:id}
  Given a matrix $A\in\mathbb{C}^{m\times |\J|}$ with columns indexed by $\J$ and a tolerance $\epsilon > 0$, an $\epsilon$-accurate interpolative decomposition of $A$ is a partitioning of $\J$ into DOF sets
  associated with so-called skeleton columns $\sk\subset \J$ and redundant columns $\rd = \J \setminus \sk$ and a corresponding interpolation matrix $\,T_\J$ such that
\begin{align*}
A(:,\rd)= A(:,\sk)T_\J +E,
\end{align*}
where $\|E\|_2 \le \epsilon\|A\|_2$.  In other words, the
redundant columns are approximated as a linear combination of the
skeleton columns to within the prescribed relative accuracy.
\end{definition}

Clearly, the ID of Definition \ref{def:id} trivially always exists by
taking $\sk = \J$.  In practice, however, we aim
to use the ID to compress $\J$ such that $|\sk|$ is close to the true $\epsilon$-numerical rank.

We do not go in detail into IDs or their computation here, but refer the reader
to \cite{id,liberty2007randomized} and references therein for a more
detailed presentation.  We do, however, note an important point in finding an ID of $A$: when constructing the ID, it is only necessary to consider nonzero rows of $A$.  This is evident from the fact that, if $\mathcal{J}$ is a set of DOFs corresponding to all the nonzero rows of $A$, then an ID of $A(\mathcal{J},:)$ yields a partitioning of $\J$ and an interpolation matrix $T_\J$ that also give a valid ID of all of $A$ to the same accuracy.

In what follows, we will use the notation
\begin{align}
\left[\sk,\, \rd,\,T_\J \right] = \texttt{id}(A,\J,\epsilon)
\end{align}
to denote functions that return the relevant pieces of an $\epsilon$-accurate ID of a matrix $A$.

\subsection{Skeletonization}
By applying the ID to numerically low-rank, off-diagonal subblocks of a matrix
$A$, we expose redundancy that can be exploited through block Gaussian
elimination to approximately sparsify $A$ via a multiplicative
procedure known as skeletonization factorization.

To begin, let $A\in\mathbb{C}^{n\times n}$ and let $\J \subset [n]$ be an index set of interest.  We compress the blocks $A(\J^c,\J)$ and $A(\J,\J^c)$ by computing an ID with tolerance $\epsilon$ to obtain
\begin{align}\label{eq:stacked}
\left[\sk,\, \rd,\,T_\J \right] = \texttt{id}\left(\left[\begin{array}{l}A(\J^c,\J)\\ A(\J,\J^c)^* \end{array} \right],\J,\epsilon\right)
\end{align}  with which we can write $A$ (up to a permutation) in block form as
\begin{align}
A  &=\left[ \begin{array}{rl|l}
A(\J^c,\J^c) & A(\J^c,\sk) & A(\J^c,\rd) \\
A(\sk,\J^c) & A(\sk,\sk) & A(\sk,\rd) \\ \hline
A(\rd,\J^c) & A(\rd,\sk) & A(\rd,\rd)
\end{array}\right] \\
&\approx
 \left[ \begin{array}{rl|l}
A(\J^c,\J^c) & A(\J^c,\sk) & A(\J^c,\sk)T_\J \\
A(\sk,\J^c) & A(\sk,\sk) & A(\sk,\rd) \\ \hline
T_\J^*A(\sk,\J^c) & A(\rd,\sk) & A(\rd,\rd)
\end{array}\right],
\end{align}
 up to relative error $\mathcal{O}(\epsilon).$  By a sequence of block row and column operations, we can eliminate the top-right and bottom-left blocks to obtain
\begin{align}\label{eq:Dblocks}
Q_\J^*AQ_\J &\approx
 \left[ \begin{array}{rl|l}
A(\J^c,\J^c) & A(\J^c,\sk) &\\
A(\sk,\J^c) & A(\sk,\sk) & D_{\sk,\rd} \\ \hline
& D_{\rd,\sk} & D_{\rd,\rd}
\end{array}\right],
\end{align}
where $Q_\J$ is given by
\begin{align}\label{eq:q}
Q_\J&= \left[\begin{array}{cc|c}I &  &  \\
                              & I &-T_\J  \\[0.125em] \hline
                              &  & I\end{array} \right]
\end{align}
and the $D$ subblocks are linear combinations of the $A$ subblocks.  We use these matrices to define the skeletonization factorization of $A$ via a final step of elimination of the DOFs $\rd$ as in \cite{hifie}.

\begin{definition}\label{def:bge}
Assume that $D_{\rd,\rd}$ in \eqref{eq:Dblocks} is nonsingular.  We define the skeletonization factorization of $A$ with respect to the DOFs $\J$ (up to a permutation) as
\begin{align*}
Z_\J(A) &\equiv  \left[ \begin{array}{rl|l}
A(\J^c,\J^c) & A(\J^c,\sk) &\\
A(\sk,\J^c) & D_{\sk,\sk} &  \\ \hline
&  & D_{\rd,\rd}
\end{array}\right]\approx M_\J^*Q_\J^* AQ_\J H_\J \equiv U_\J^*AV_\J,
\end{align*}
where the Schur complement block is $D_{\sk,\sk} = A_{\sk,\sk} - D_{\sk,\rd}D_{\rd,\rd}^{-1}D_{\rd,\sk}$, the matrices $M_\J^*$ and $H_\J$ are given by
\begin{align*}
M_\J^* \equiv \left[\begin{array}{cc|c}I &  &  \\
                              & I & -D_{\sk,\rd}D_{\rd,\rd}^{-1} \\[0.25em] \hline
                              &  & I\end{array} \right], \hspace{0.5cm} H_\J \equiv \left[\begin{array}{cc|c}I &  &  \\
                              & I &  \\ \hline &&\\[-1.0em]
                              & -D_{\rd,\rd}^{-1}D_{\rd,\sk}& I\end{array} \right],
\end{align*}
and the remaining matrices are as defined in \eqref{eq:Dblocks} and \eqref{eq:q}.
\end{definition}

We note that in $Z_\J(A)$ the redundant DOFs $\rd$ have been
completely decoupled from the rest while leaving the off-diagonal
interactions $A_{\J^c,\sk}$ and $A_{\sk,\J^c}$ unchanged.  Thus, after
this skeletonization process we will refer to the skeleton DOFs
$\sk$ as \emph{active} and the redundant DOFs $\rd$
as \emph{inactive}.  Henceforth, we refer to this skeletonization
factorization as ``skeletonization'', though this is not to be
confused with the sense in which the term is used in \cite{id}.

Clearly the matrices $U_\J$ and $V_\J$ are highly structured since they are each the product of block unit-triangular matrices.  As such, we will write the skeletonization of a matrix $A$ with respect to the DOFs $\J$ with accuracy $\epsilon$ as
\begin{align}
\left[\sk, \,\rd,\,D_{\sk,\sk},\, D_{\rd,\rd},\, U_\J,\, V_\J \right] = \texttt{skel}(A,\J,\epsilon),
\end{align}
where $U_\J$ and $V_\J$ are understood to be stored as a product of operators in block form that can be applied and inverted cheaply and clearly one can construct $Z_\J(A)$ implicitly from the information returned by $\texttt{skel}(A,\J,\epsilon)$.

\subsection{Group skeletonization}\label{sec:group}

Notationally, it will be useful as in \cite{hifie} to extend the notion of skeletonization of a matrix $A$ with respect to an index set $\J$ to skeletonization of $A$ with respect to multiple disjoint index sets.  In particular, for two index sets $\J$ and $\J'$ with $\J\cap\J' = \emptyset$ we can perform the independent skeletonizations
\begin{align}
\left[\sk,\, \rd,D_{\sk,\sk},\, D_{\rd,\rd},\, U_\J,\, V_\J \right] &= \texttt{skel}(A,\J,\epsilon)\\
\left[\sk',\, \rd',D_{\sk',\sk'},\, D_{\rd',\rd'},\, U_{\J'},\, V_{\J'} \right] &= \texttt{skel}(A,\J',\epsilon),
\end{align}
whereupon we observe that the matrices $U_{\J}$ and $U_{\J'}$ (and similarly $V_{\J}$ and $V_{\J'}$) commute, which motivates us to define the group skeletonization of $A$ with respect to the index sets $\J$ and $\J'$ as
\begin{align}
Z_{\{\J,\J'\}}(A) \approx U_{\J'}^*U_\J^*AV_\J V_{\J'} = U_\J^*U_{\J'}^*AV_{\J'} V_\J,
\end{align}
where we understand the approximation to be in the same sense as in Definition \ref{def:bge}, \ie, the remainder matrix $E$ from each ID is assumed to be zero such that the off-diagonal blocks of $A$ indexed by $\rd$ can be exactly eliminated by blocks indexed by $\sk$ (and likewise for $\rd'$ and $\sk'$).  By construction, the remainder error matrix is small and therefore ignored.

More generally, given a pairwise-disjoint collection of index sets $\mathscr{C} = \{\J_1,\dots,\J_m \}$ with each $\J_i \subset [n]$, we similarly define the simultaneous group skeletonization of $A$ with respect to $\mathscr{C}$ as
\begin{align}\label{eq:groupskel}
  Z_\mathscr{C}(A) \approx U_\mathscr{C}^*AV_\mathscr{C} = \left(\prod_{\J\in \mathscr{C}}U_\J^*\right)A\left(\prod_{\J\in \mathscr{C}}V_\J\right) ,
\end{align}
with $U_\mathscr{C}$ and $V_\mathscr{C}$
understood to be stored as relevant blocks of their constituent matrices.

\subsection{Acceleration using equivalent interactions}\label{sec:proxy}
In principle, the $D$ blocks in \eqref{eq:Dblocks} and Definition $\ref{def:bge}$ depend on all of $A(:,\J)$ and $A(\J,:)$ due to the fact that they depend on the partitioning and interpolation matrix coming from the ID of off-diagonal blocks of $A$.  In practice, however, we are not considering a general matrix $A$ but rather the explicit matrix $G$ in \eqref{eq:linear} coming from the discretization of the integral of some elliptic kernel.  Since such kernels frequently satisfy some form of Green's theorem wherein the
values of the kernel inside a domain can be recovered from
those on the boundary, a key trick for
reducing algorithmic complexity and increasing locality that is common
in the literature is the use of an equivalent \emph{proxy
surface}, see, \eg, \cite{id,corona2013,domainsAd,gg,rskel,martinsson-rokhlin,pan,kifmm,hifie}.

Let $\J_b$ correspond to the DOF set associated with a single leaf box $b$ at level $L$ of our quadtree, such that the complement DOF set $\J_b^c = [n]\setminus \J_b$ contains DOFs corresponding to all other leaf boxes.  As seen in Figure \ref{fig:proxycircle}, we can draw a smooth proxy surface $\Gamma^\text{prox}_{\J_b}$ around $b$ such that only the leaf boxes immediately adjacent to $b$ in the quadtree intersect the interior of the proxy surface.  By choosing a small number of points to discretize $\Gamma^\text{prox}_{\J_b}$, we can write down the matrix $P_{\J_b}$ corresponding to discretized kernel interactions between $\J_b$ and the proxy points.
Then, letting $\N_b$ refer to the collection of DOFs in leaf boxes adjacent to $b$ and $\F_b = \J_b^c\setminus \N_b$, we know that in the continuous limit, there exists a bounded linear operator $W$ such that can write (up to a permutation)
\begin{align}
G(\J_b^c,\J_b) &= \left[\begin{array}{l} G(\N_b,\J_b) \\ G(\F_b,\J_b) \end{array} \right] \approx \left[\begin{array}{cc}I & \\ & W\end{array} \right]\left[\begin{array}{c} G(\N_b,\J_b) \\ P_{\J_b} \end{array} \right].
\end{align}
 With this
representation, we can now perform an ID of simply the right-most
matrix above and use this to obtain an ID of $G(\J_b^c,\J_b)$.  This is desirable for two reasons.  First, since the interpolation matrix $W$ potentially has many more rows than
columns, computing an ID of this surrogate matrix can be much less
computationally expensive. Second, we see that performing an ID with respect to $\J_b$ is now dependent only on DOFs in the boxes immediately adjacent to $b$, which increases locality and will be essential for our fast updating algorithm.  For a more thorough treatment of the proxy surface, see \cite{hifie}.

\begin{figure}
\centering
\begin{subfigure}[c]{0.4\textwidth}
\centering
\scalebox{0.65}{
  \begin{tikzpicture}
  \foreach \x in {2,...,4}
      \foreach \y in {2,...,4}
          \filldraw[lightgray] (\x,\y) rectangle (\x+1, \y+1);

  \filldraw[darkgray,thick,draw=black] (3,3) rectangle (4,4);
  \draw[step=1cm,black,dashed,thick] (0.5,0.5) grid (6.5,6.5);
  \draw[black,dotted, very thick]  (3.5, 3.5) circle(1.5);
  \draw[black, thin]  (3.5, 3.5) circle(1.5);
  \end{tikzpicture}
  }
  \caption{\label{fig:proxycircle}\vspace*{-.4cm}}
  \end{subfigure}
  \
  \begin{subfigure}[c]{0.4\textwidth}
  \centering
  \scalebox{0.65}{
      \begin{tikzpicture}
      \draw[step=1cm,black,dashed,thick,shift={(-3.5,0.5)}] (0.5,0.5) grid (6.5,6.5);
            \draw[draw=none, use as bounding box, clip,shift={(-3.5,0.5)}](1,1) rectangle (6,6);
      \draw[step=1cm,black,dashed,thick,shift={(-3.5,0.5)}] (0.5,0.5) grid (6.5,6.5);
      \draw[step=0.707cm, thick,darkgray,rotate=45] (-0.5,-0.5) grid (6.5,6.5);
 \end{tikzpicture}
 }
 \caption{\label{fig:voro}\vspace*{-.4cm}}
  \end{subfigure}
  \caption{(a): Using the dotted black circle as a proxy surface when skeletonizing DOFs in the dark gray box, only interactions between that box and the light gray boxes adjacent to it need be considered.  In particular, all interactions between the dark gray box and the white boxes will be represented by equivalent interactions using the proxy surface.   (b): If the dashed black grid corresponds to edges of boxes at level $\ell$, then the assignment of DOFs to edges is given using the Voronoi tessellation about the edge centers (gray rotated grid).\vspace*{-.6cm}}
\end{figure}
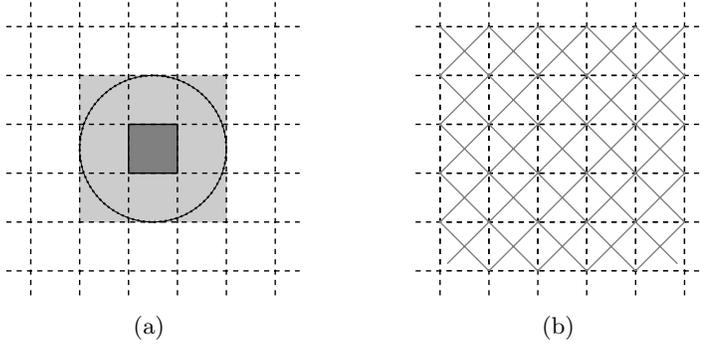
\section{Factorization algorithms}\label{sec:factor}
Here we describe the recursive skeletonization factorization and hierarchical interpolative factorization detailed in \cite{hifie} in a manner conducive to discussing efficient updating.  We remark that our notation here differs in several ways from the previous presentation.
\subsection{Recursive skeletonization factorization}
With our notion of group skeletonization using the proxy trick described in Section \ref{sec:proxy}, we construct the recursive skeletonization
factorization of \cite{hifie} described in Algorithm \ref{alg:rskelf},
which is an alternative approach to the hierarchical compression
scheme in \cite{domainsAd, martinsson-rokhlin,rskel} for
solving \eqref{eq:linear}.  The resulting multiplicative factorization
closely resembles a variant of the so-called ULV
decomposition in \cite{fastulv} while taking advantage of intermediate
reduced representations of matrix blocks such as in \cite{xia2} to
attain better complexity.

As previously stated, we hierarchically decompose our domain $\Omega$ using a quadtree with root level $\ell=0$ and lowest level $\ell=L$.  As in Section \ref{sec:proxy}, we will consider the DOFs corresponding to a leaf box $b$, $\J_b$, to be the set of DOFs corresponding to discretization points interior to $b$.  Then, we define the collection of DOF sets for level $L$ as
\begin{align}\label{eq:ell}
\L_L = \{ \J_b \, \vert\, b \text{ is a box at level $L$} \}.
\end{align}

Using the group skeletonization process described in Section \ref{sec:group} to skeletonize $G$ in \eqref{eq:linear} with respect to the collection $\L_L$ with tolerance $\epsilon$ yields the sparsified matrix $Z_{\L_L}(G)$.  We functionally write this group skeletonization as
\begin{align}
\left[\left\{\sk_b,\, \rd_b,\, D_{\sk_b,\sk_b},\, D_{\rd_b,\rd_b}\right\}_{\J_b \in \L_L},\,U_{\L_L},\, V_{\L_L} \right] = \texttt{skel}(G,\L_L,\epsilon).
\end{align}

After group skeletonization at level $L$, we move to level $\ell=L-1$.  For a leaf box $b$, we define the DOFs $\J_b$ as before, but now we also have some boxes $b$ at level $L-1$ that are not leaves.  For such boxes $b$, we will define the child set $\child(b)$ to be the set of child boxes of $b$ in the quadtree and define the DOFs associated with $b$ as the set of active DOFs corresponding to children of $b$, \ie,
\begin{align}\label{eq:activedofs}
\J_b &= \bigcup_{b' \in \child(b)} \sk_{b'}.
\end{align}
With this definition, we can define $\L_\ell$ for any level $\ell < L$ analogously to \eqref{eq:ell} and again skeletonize with respect to $\L_\ell$ to further eliminate DOFs.  Repeating this process level by level constitutes the recursive skeletonization factorization, \rskelf{}, as made concrete in Algorithm \ref{alg:rskelf}.

 \begin{algorithm}
 \small
  \caption{Recursive skeletonization factorization (\rskelf{})}
  \label{alg:rskelf}
  \begin{algorithmic}
   \State $G_L = G$
   \For{$\ell = L, L-1, \dots, 1$}
   \State{\texttt{// get skeleton blocks and operators}}
   \State $\left[\left\{\sk_b,\, \rd_b,\, D_{\sk_b,\sk_b},\, D_{\rd_b,\rd_b}\right\}_{\J_b \in \L_\ell},\,U_{\L_\ell},\, V_{\L_\ell}\right] = \texttt{skel}(G_\ell,\L_\ell,\epsilon)$
   \State{\texttt{// assemble skeletonization}}
   \State $G_{\ell-1} = G_\ell$
    \For{$\J_b \in \L_\ell$ with $\J_b=\sk_b\cup\rd_b$}
    \State $G_{\ell-1}(:,\rd_b) = G_{\ell-1}(\rd_b,:) = 0$
    \State $G_{\ell-1}(\sk_b,\sk_b) = D_{\sk_b,\sk_b}$
    \State $G_{\ell-1}(\rd_b,\rd_b) = D_{\rd_b,\rd_b}$
    \EndFor
   \EndFor
   \State $G \approx F \equiv U_{\L_L}^{-*} \cdots U_{\L_1}^{-*} G_0 V_{\L_1}^{-1} \cdots V_{\L_L}^{-1}$
  \end{algorithmic}
 \end{algorithm}

As before, we note that the large matrices $U_{\L_\ell}$ and $V_{\L_\ell}$ for each level are purely notational and are stored in block form.  In fact, even explicit assembly of $G_{\ell-1}$ is not strictly necessary but written purely for exposition.

At each level $\ell$ in Algorithm \ref{alg:rskelf}, we identify for each box $b$ a set of redundant DOFs $\rd_b$ which are completely decoupled from the rest of the DOFs as evidenced by the zero blocks introduced into $G_{\ell-1}$.  Therefore, as observed in Section \ref{sec:id}, it is not necessary to consider the redundant DOFs from any level $\ell' > \ell$ when skeletonizing level $\ell$.  Further, the use of the proxy trick described in Section \ref{sec:proxy} implies that when skeletonizing box $b$ we need to consider only the set of neighboring active DOFs
\begin{align}\label{eq:nbor}
\N_b &= \bigcup_{b' \in \nbor(b)} \J_{b'},
\end{align}
where $\nbor(b)$ is the function that maps a box to the collection of adjacent boxes that are either also on level $\ell$ or are on level $\ell' < \ell$ and have no children.  This second criterion serves to address the case of heterogeneous tree refinement.  With the use of the proxy surface, skeletonization requires only local matrix operations with cost $\mathcal{O}(|\J_b|^3)$.  This combined with \eqref{eq:activedofs} shows that the cost of \rskelf{} depends strongly on the scaling of $|\sk_b|$ across all boxes with respect to $N$.

We let $|s_\ell|$ refer to the average number of skeleton DOFs per box at level $\ell$, \ie,
 \begin{align}
 |s_\ell| = \frac{1}{|\L_\ell|}\sum_{\substack{\J_b\in\L_\ell\\ \J_b = \sk_b\cup\rd_b}} | \sk_b|,
 \end{align}
 and note that $|s_0|=0$ by this definition. We see in Figure \ref{fig:rskelf} that the skeleton DOFs tend to cluster near the boundaries of the quadtree boxes, such that for the quasi-1D problem at the top of the figure we have $|s_1|$ essentially independent of $N$ for elliptic kernels as documented in \cite{rskel,hifie}, leading to an asymptotic cost of $\mathcal{O}(N)$ for construction of $G$.  However, as seen in the bottom of the same figure, in true 2D problems the clustering of DOFs near box boundaries for the same kernel results in $|s_1|$ scaling as $\mathcal{O}(N^{1/2}),$ making factorization with \rskelf{} asymptotically more expensive and thus necessitating modifications as described in Section \ref{ref:hif}.

 \begin{figure}
  \centering
  \begin{subfigure}{0.2\textwidth}
  \centering
   \includegraphics[width=0.7\textwidth]{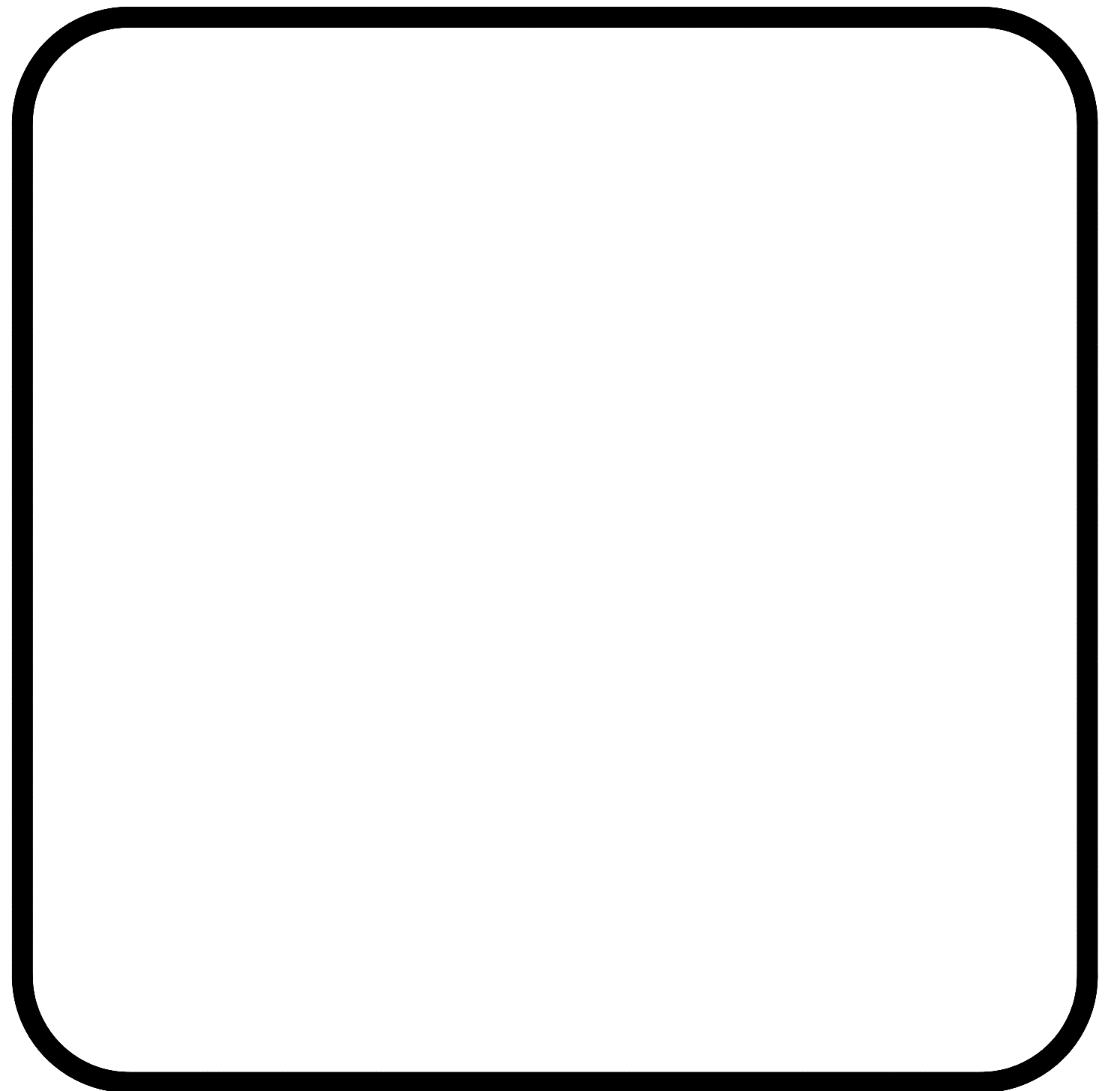}
   \caption*{$\ell = 3$}
  \end{subfigure}
  \qquad
  \begin{subfigure}{0.2\textwidth}
  \centering
   \includegraphics[width=0.7\textwidth]{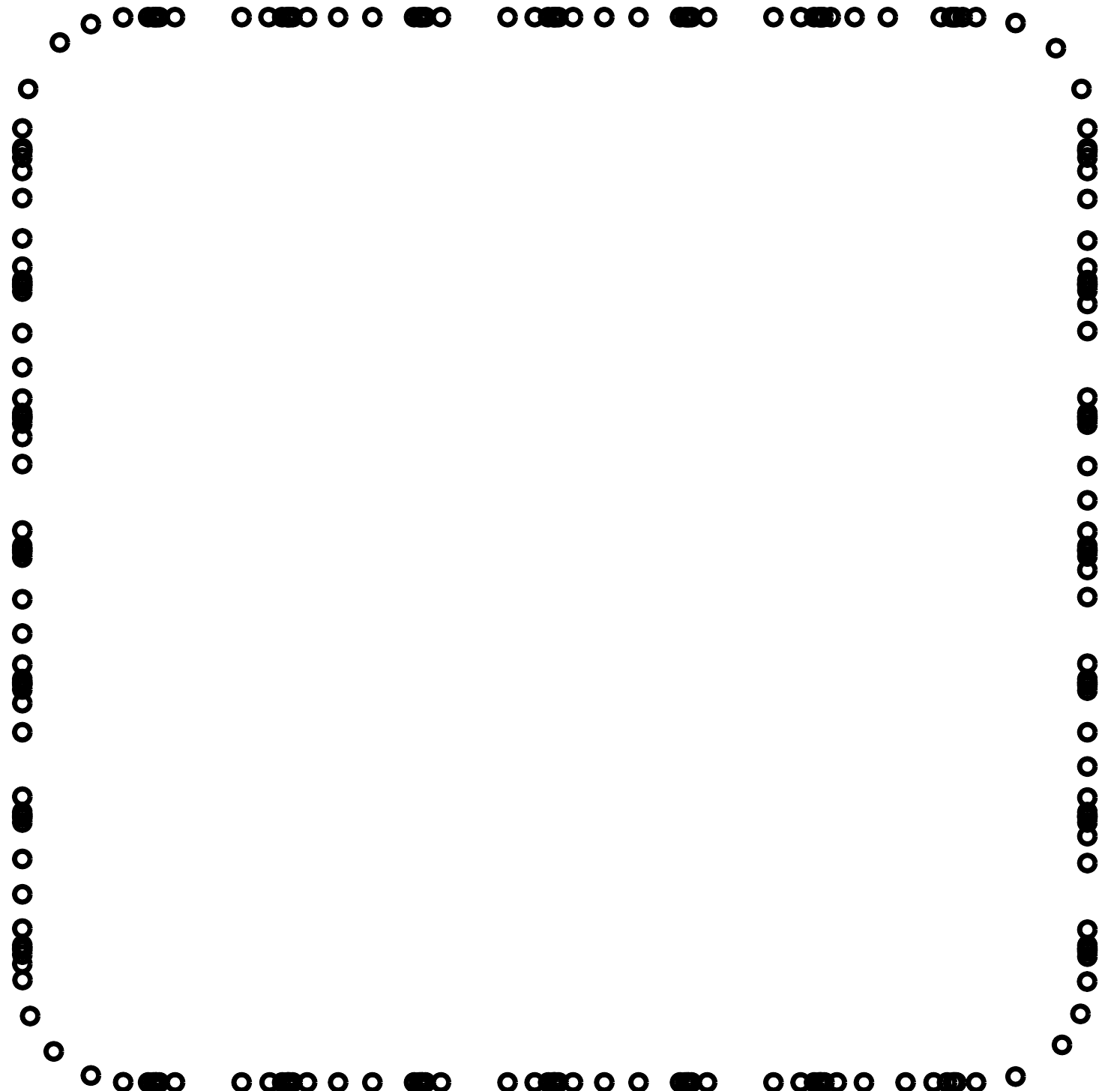}
   \caption*{$\ell = 2$}
  \end{subfigure}
  \qquad
  \begin{subfigure}{0.2\textwidth}
  \centering
   \includegraphics[width=0.7\textwidth]{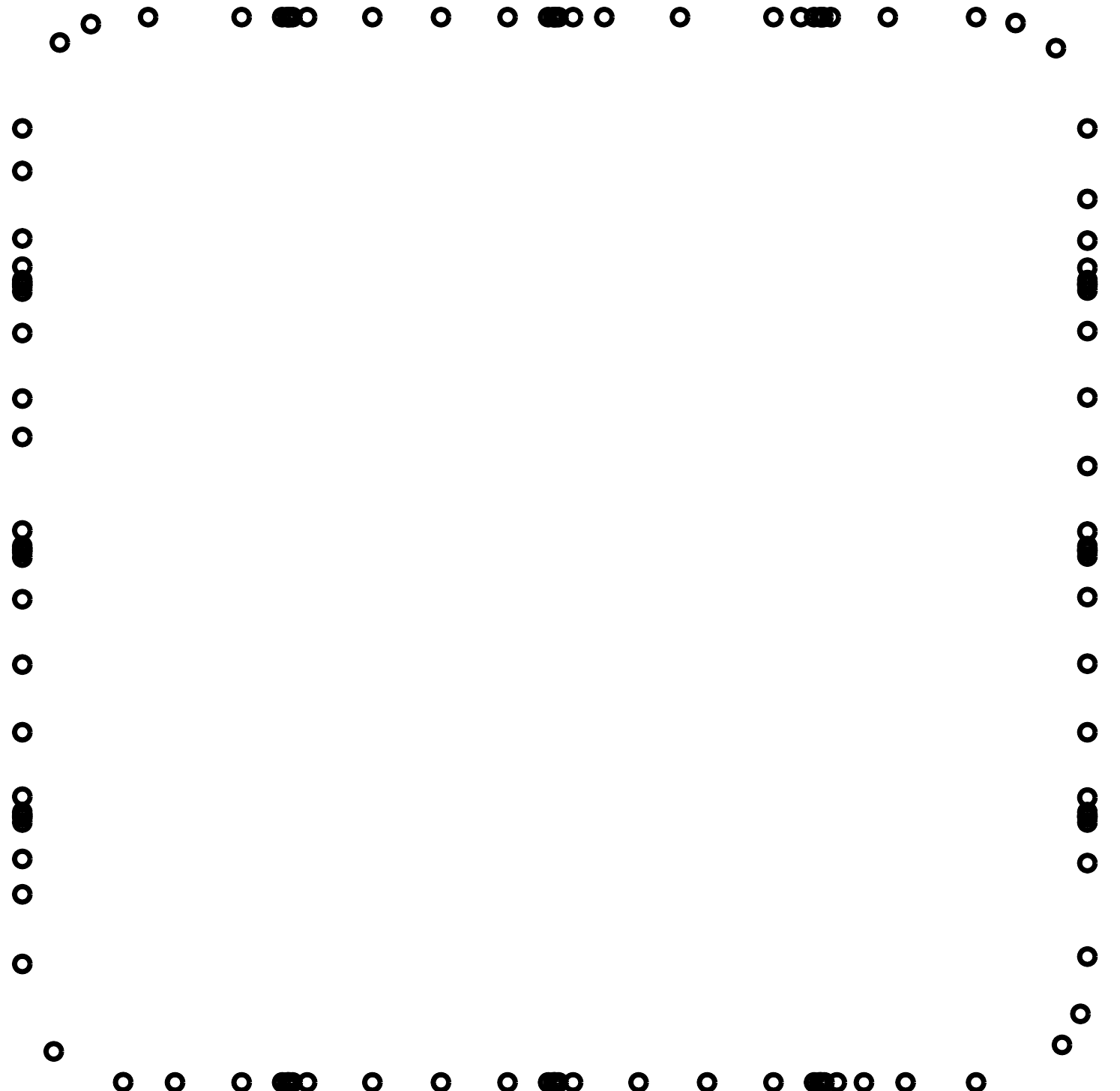}
   \caption*{$\ell = 1$}
  \end{subfigure}
  \qquad
  \begin{subfigure}{0.2\textwidth}
  \centering
   \includegraphics[width=0.7\textwidth]{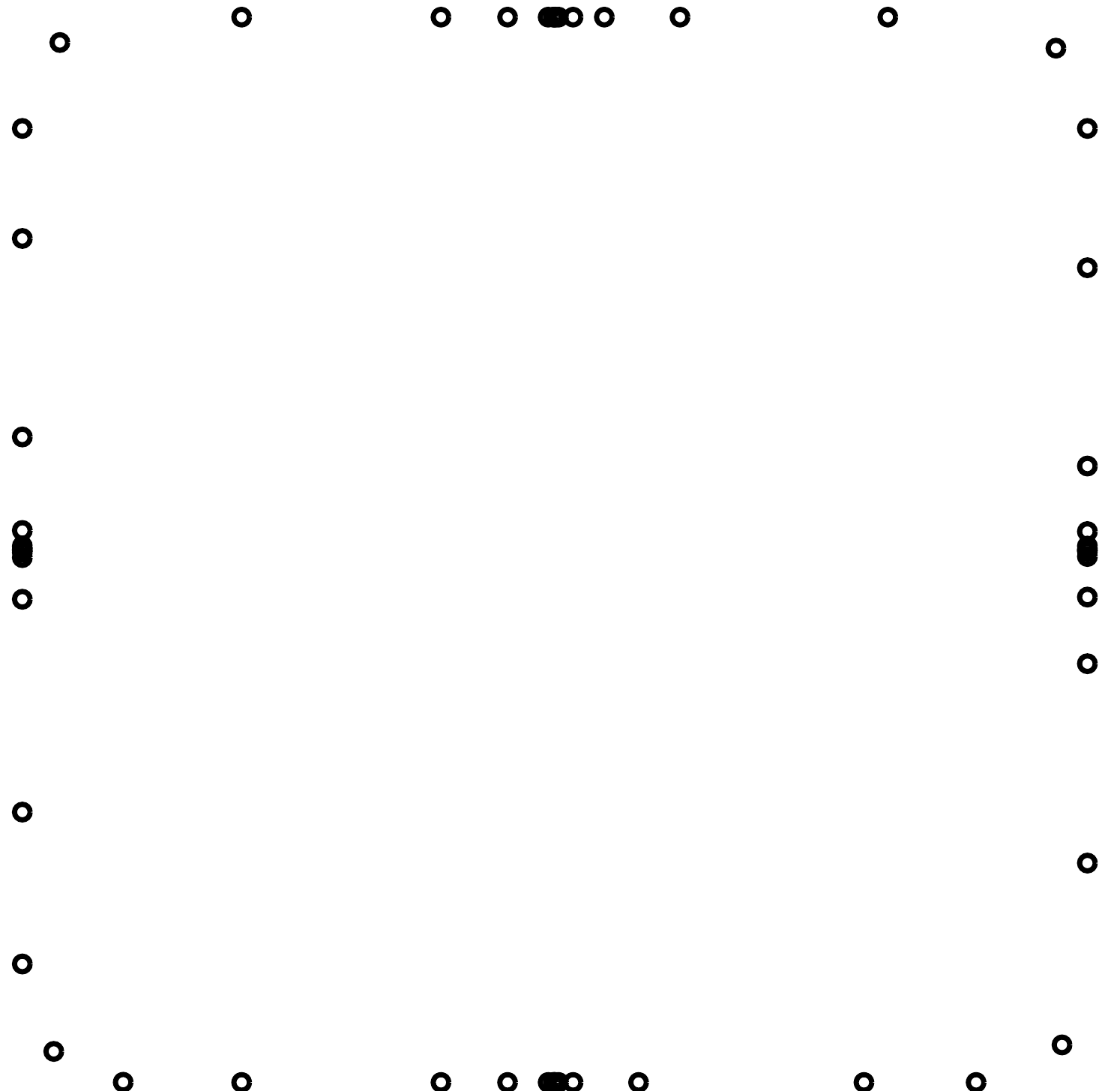}
   \caption*{$\ell = 0$}
  \end{subfigure}

    \begin{subfigure}{0.2\textwidth}
    \centering
   \includegraphics[width=0.7\textwidth]{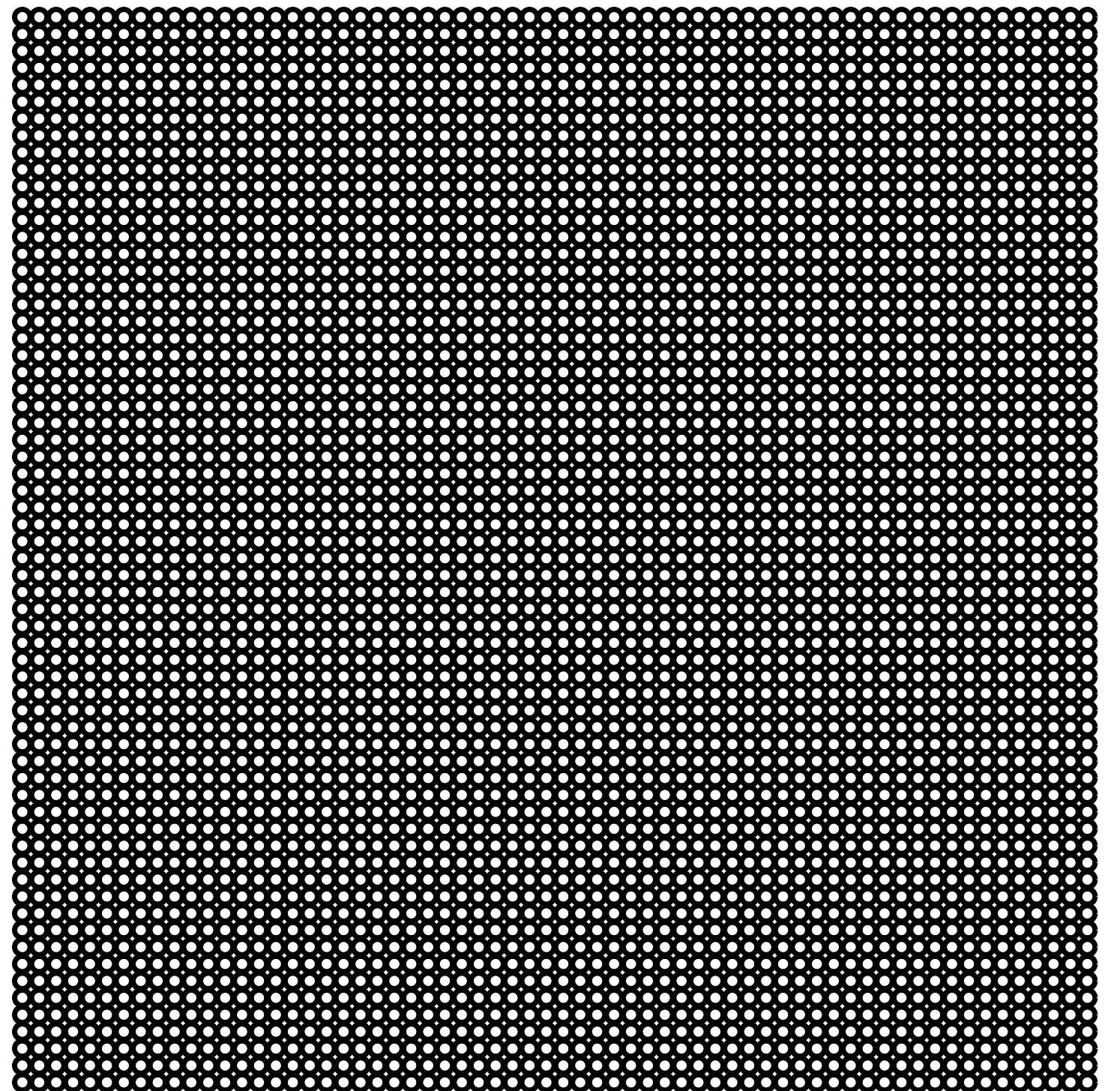}
   \caption*{$\ell = 3$\vspace*{-.4cm}}
  \end{subfigure}
  \qquad
  \begin{subfigure}{0.2\textwidth}
  \centering
   \includegraphics[width=0.7\textwidth]{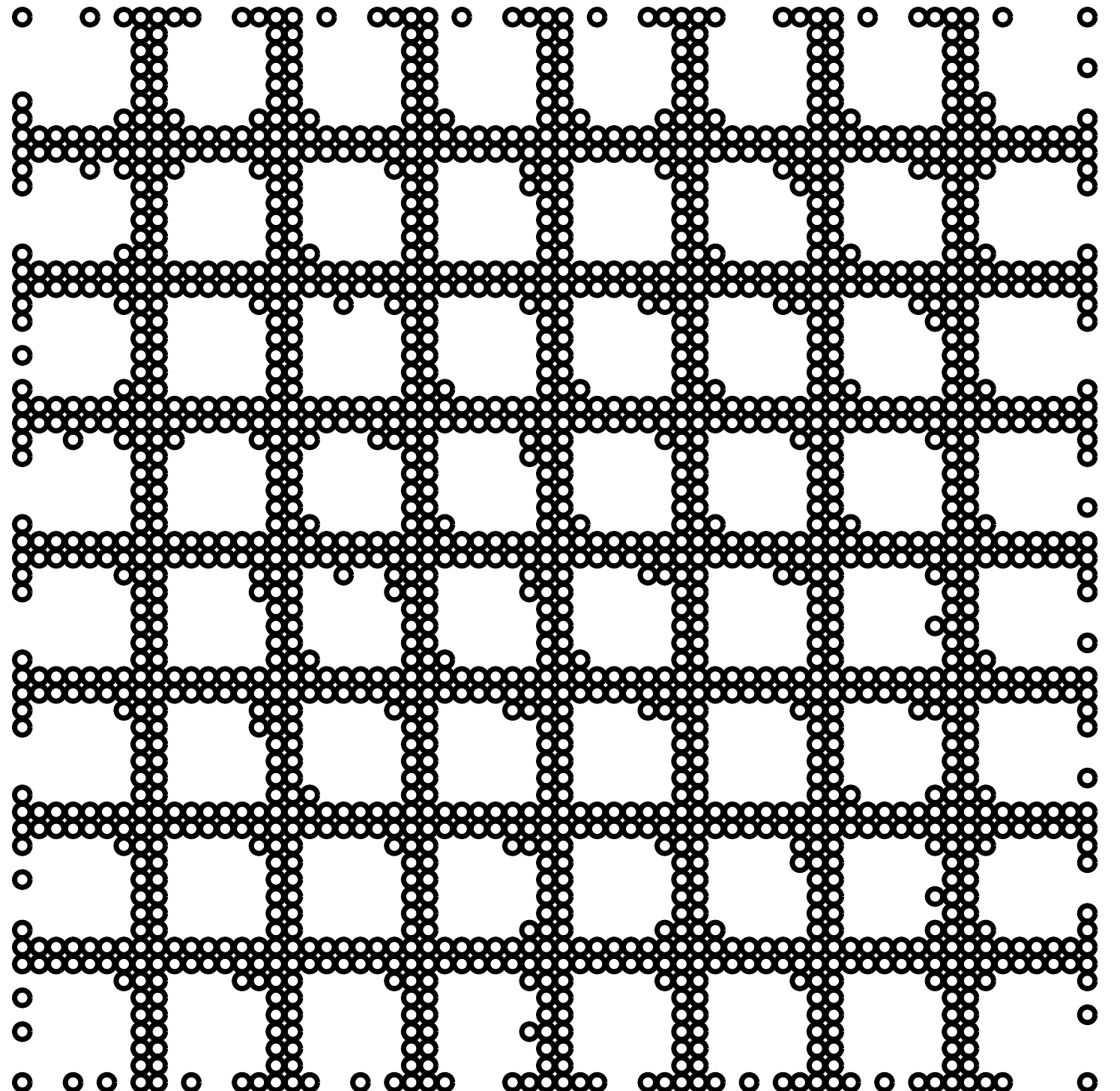}
   \caption*{$\ell = 2$\vspace*{-.4cm}}
  \end{subfigure}
  \qquad
  \begin{subfigure}{0.2\textwidth}
  \centering
   \includegraphics[width=0.7\textwidth]{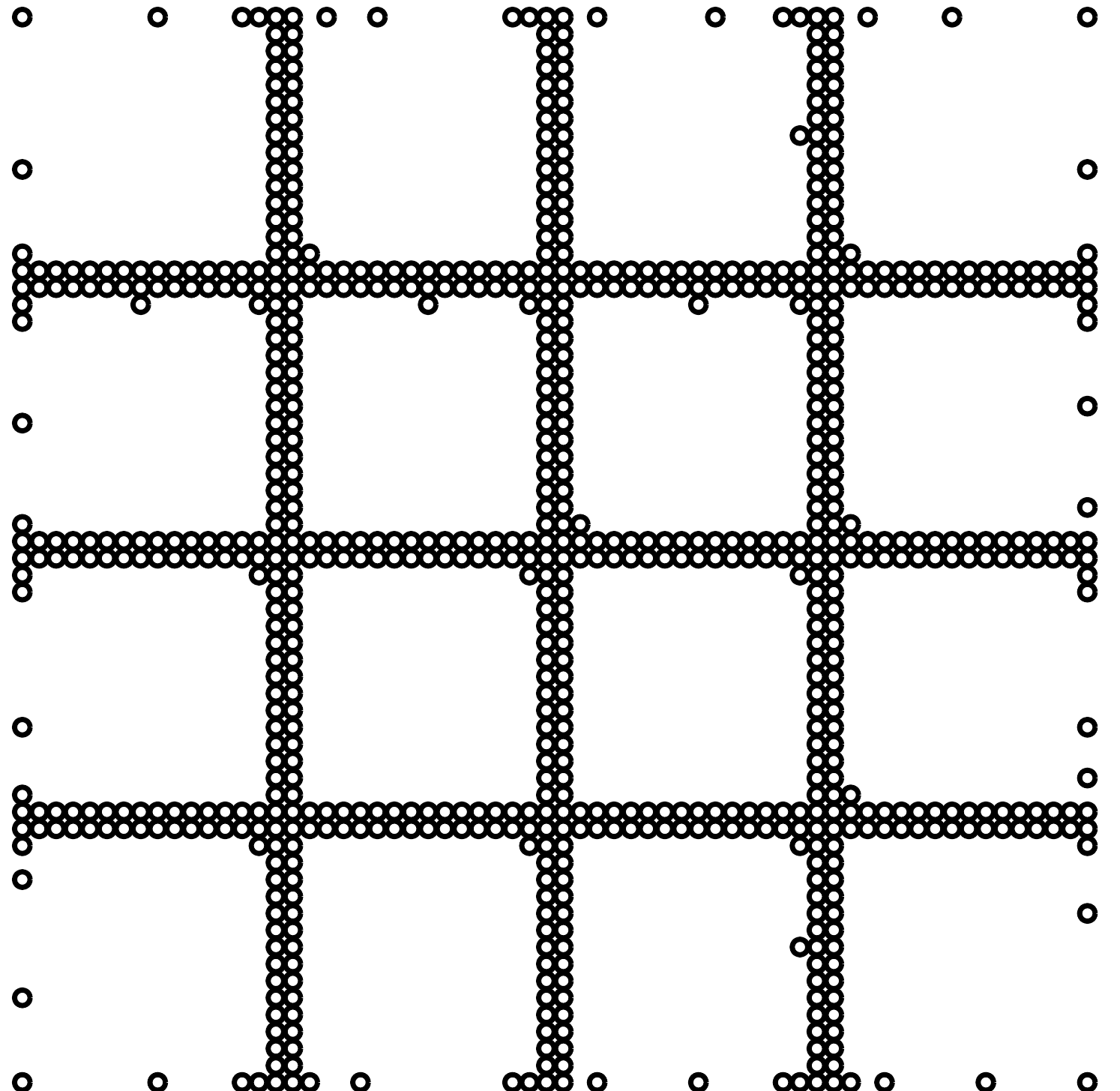}
   \caption*{$\ell = 1$\vspace*{-.4cm}}
  \end{subfigure}
  \qquad
  \begin{subfigure}{0.2\textwidth}
  \centering
   \includegraphics[width=0.7\textwidth]{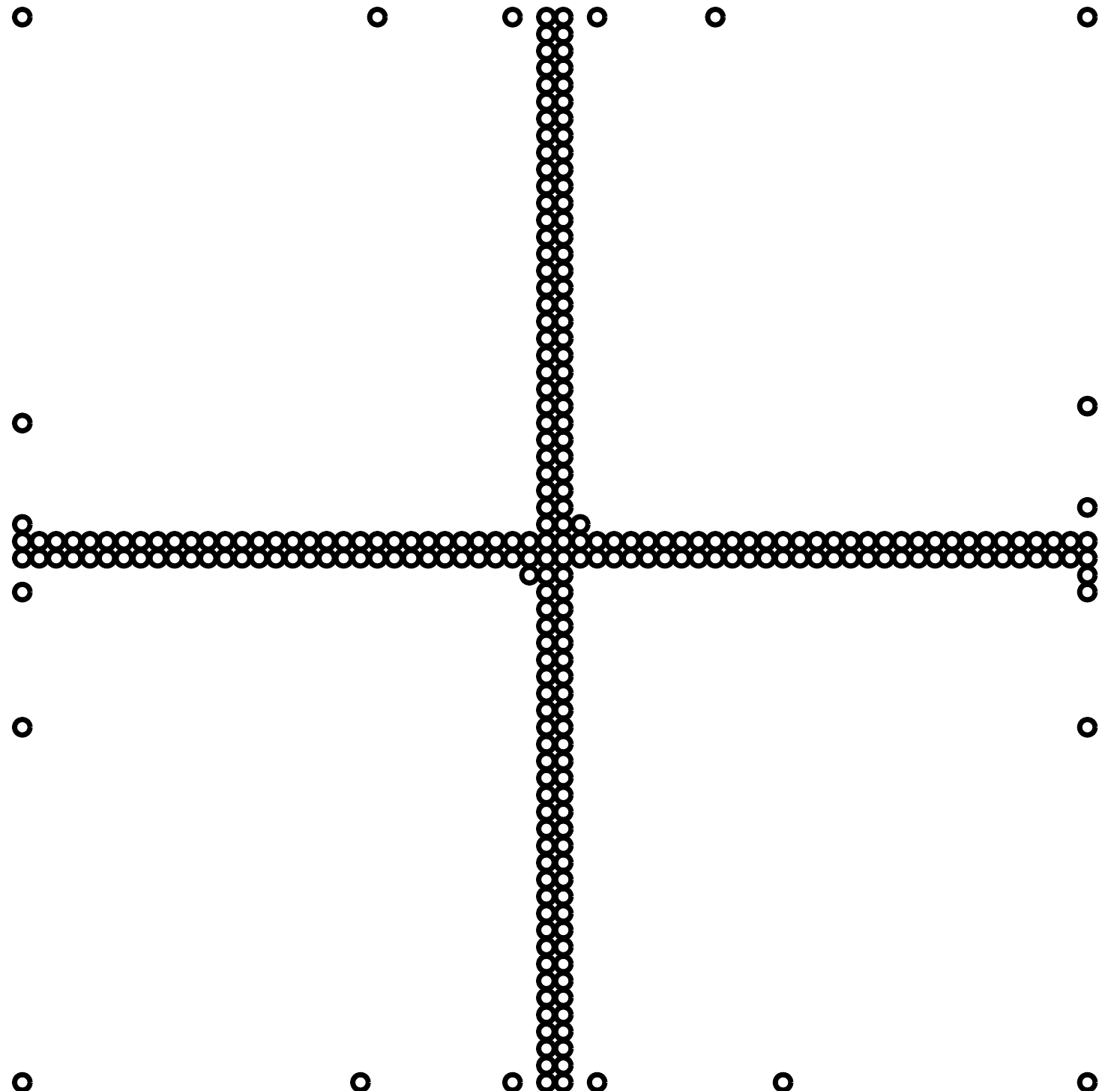}
   \caption*{$\ell = 0$\vspace*{-.4cm}}
  \end{subfigure}

  \caption{Active DOFs before skeletonizing each level $\ell$ of \rskelf{} on a quasi-1D problem (top) and true 2D problem (bottom).  We see that the DOFs cluster near the edges of the boxes of the quadtree at each level. \vspace*{-.6cm}}
  \label{fig:rskelf}
 \end{figure}

\subsection{The hierarchical interpolative factorization}\label{ref:hif}
For true 2D problems, complexity estimates give that the construction of $F$ in Algorithm \ref{alg:rskelf} costs roughly $\mathcal{O}(N^{3/2})$, as seen in the previous section. To recover linear complexity for this case, we use the \emph{hierarchical interpolative factorization} (\hif{}) as described in \cite{hifie}, which is based on the same fundamental operations as \rskelf{} but adds an extra step of skeletonization between quadtree levels.

In 2D, \hif{} proceeds as follows.  We begin just as in \rskelf{} by assuming a quadtree decomposition of space, defining $\L_L$ as in \eqref{eq:ell}, and skeletonizing $G$ with respect to $\L_L$ to obtain $\skel_{\L_L}(G)$.  At this point, rather than going directly to level $\L_{L-1},$ for each box $b$ we consider the four edges of $b$, $\edge(b)$, and perform a Voronoi tessellation of space with respect to the centers of all such edges as in Figure \ref{fig:voro} yielding the DOFs $\voro_e\subset[n]$ that are closest to $e$ in the tessellation.  We consider these edges to be part of the half-integer level $L-1/2$, and for every such edge $e$ we define the associated DOFs to be the active DOFs of its two adjacent boxes that fall within the corresponding Voronoi cell, \ie,
\begin{align}\label{eq:ellhalf}
\J_e &= \{\sk_b\cap\voro_e \, \vert \, e \in \edge(b)\}\cup \{\J_{b'}\cap\voro_e \, \vert \, e \in \edge(b),\, b'\in\N_b \text { has no children }\}.
\end{align}
Note that the second collection above is necessary to capture active DOFs at higher levels in the case of hetereogeneous refinement.
Defining the level $\L_{L-1/2}$ via $\L_{L-1/2} = \{\J_e \, \vert \, e\in\edge(b), b \text{ is a box on level $L$}\}$, we peform group skeletonization with respect to $\L_{L-1/2}$ which decouples (makes inactive) additional DOFs.

At this point, we move up to the next box level $L-1$, which requires a modification of the definition of $\J_b$ in \eqref{eq:activedofs} for non-leaf boxes $b$ at this level.  In particular, since additional DOFs are no longer active, we define $\J_b$ for \hif{} as
\begin{align}
\J_b =\bigcup_{\substack{b'\in\child(b)\\e\in\edge(b')}} \sk_b\cap\sk_e,
\end{align}
\ie, only the active DOFs of children of $b$ that are still active after edge skeletonization.  We continue by alternating between skeletonizing boxes and skeletonizing edges as summarized in Algorithm \ref{alg:hifie}.

\begin{algorithm}
\small
  \caption{Hierarchical interpolative factorization (\hif{})}
  \label{alg:hifie}
  \begin{algorithmic}
     \State $G_L = G$
   \For{$\ell = L, L-1, \dots, 1$}
   \State{\texttt{// get skeleton blocks and operators for boxes}}
   \State $\left[\left\{\sk_b,\, \rd_b,\, D_{\sk_b,\sk_b},\, D_{\rd_b,\rd_b}\right\}_{\J_b \in \L_\ell},\,U_{\L_\ell},\, V_{\L_\ell}\right] = \texttt{skel}(G_\ell,\L_\ell,\epsilon)$
   \State{\texttt{// assemble skeletonization for boxes}}
   \State $G_{\ell-1/2} = G_\ell$
    \For{$\J_b \in \L_\ell$ with $\J_b=\sk_b\cup\rd_b$}
    \State $G_{\ell-1/2}(:,\rd_b) = G_{\ell-1}(\rd_b,:) = 0$
    \State $G_{\ell-1.2}(\sk_b,\sk_b) = D_{\sk_b,\sk_b}$
    \State $G_{\ell-1/2}(\rd_b,\rd_b) = D_{\rd_b,\rd_b}$
    \EndFor
    \State{\texttt{// get skeleton blocks and operators for edges}}
   \State $\left[\left\{\sk_e,\, \rd_e,\, D_{\sk_e,\sk_e},\, D_{\rd_e,\rd_e}\right\}_{\J_e \in \L_{\ell-1/2}},\,U_{\L_{\ell-1/2}},\, V_{\L_{\ell-1/2}}\right] = \texttt{skel}(G_{\ell-1/2},\L_{\ell-1/2},\epsilon)$
   \State{\texttt{// assemble skeletonization for edges}}
   \State $G_{\ell-1} = G_{\ell-1/2}$
    \For{$\J_e \in \L_{\ell-1/2}$ with $\J_e=\sk_e\cup\rd_e$}
    \State $G_{\ell-1}(:,\rd_e) = G_{\ell-1}(\rd_e,:) = 0$
    \State $G_{\ell-1}(\sk_e,\sk_e) = D_{\sk_e,\sk_e}$
    \State $G_{\ell-1}(\rd_e,\rd_e) = D_{\rd_e,\rd_e}$
    \EndFor
   \EndFor
   \State $G \approx F \equiv U_{\L_L}^{-*}U_{\L_{L-1/2}}^{-*} \cdots U_{\L_1}^{-*}U_{\L_{1/2}}^{-*} G_0 V_{\L_{1/2}}^{-1}V_{\L_1}^{-1} \cdots V_{\L_{L-1/2}}^{-1}V_{\L_L}^{-1}$
  \end{algorithmic}
 \end{algorithm}

 \begin{figure}
  \centering
  \begin{subfigure}{0.2\textwidth}
  \centering
   \includegraphics[width=0.7\textwidth]{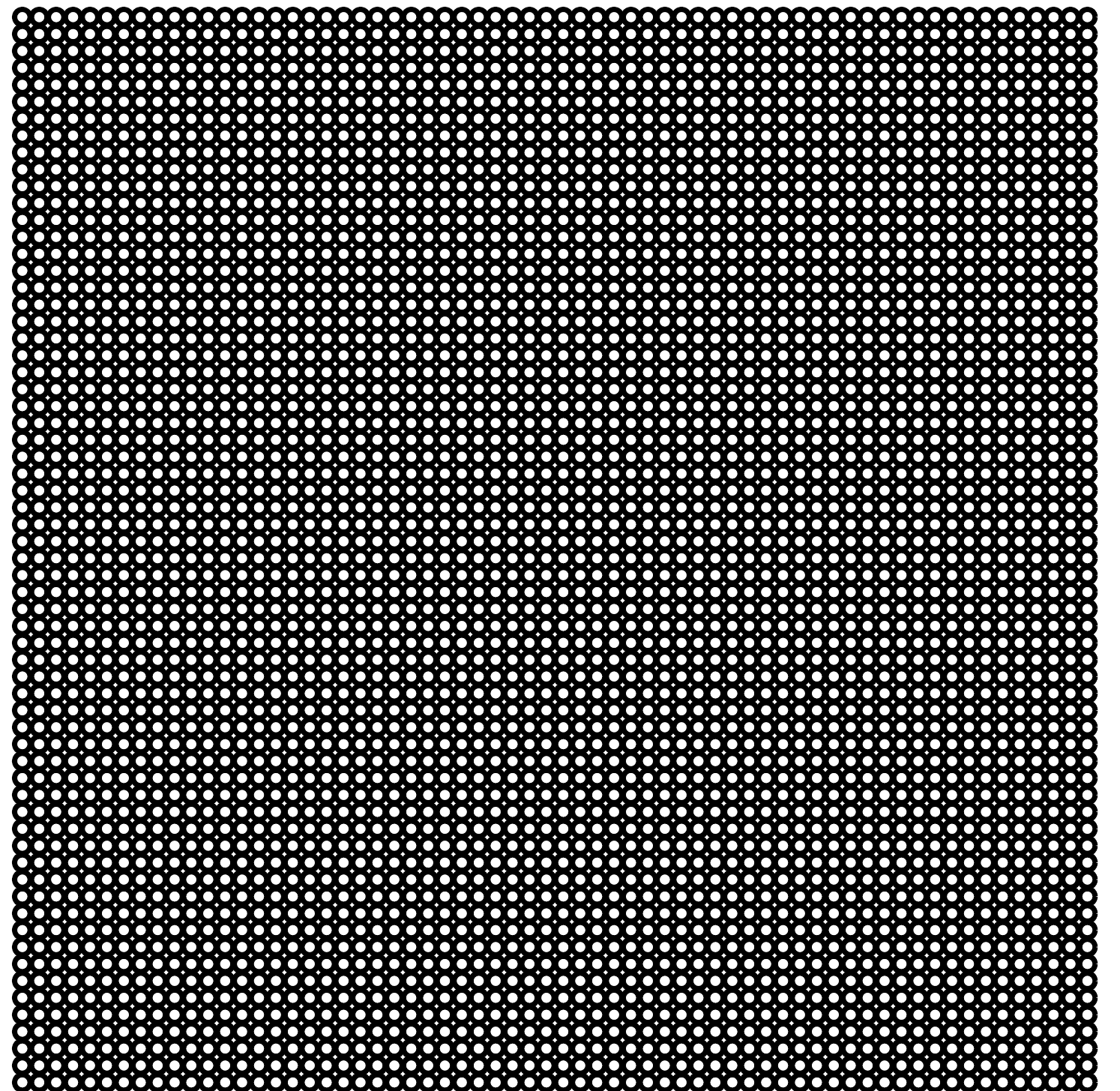}
   \caption*{$\ell = 3$\vspace*{-.4cm}}
  \end{subfigure}
  \qquad
  \begin{subfigure}{0.2\textwidth}
  \centering
   \includegraphics[width=0.7\textwidth]{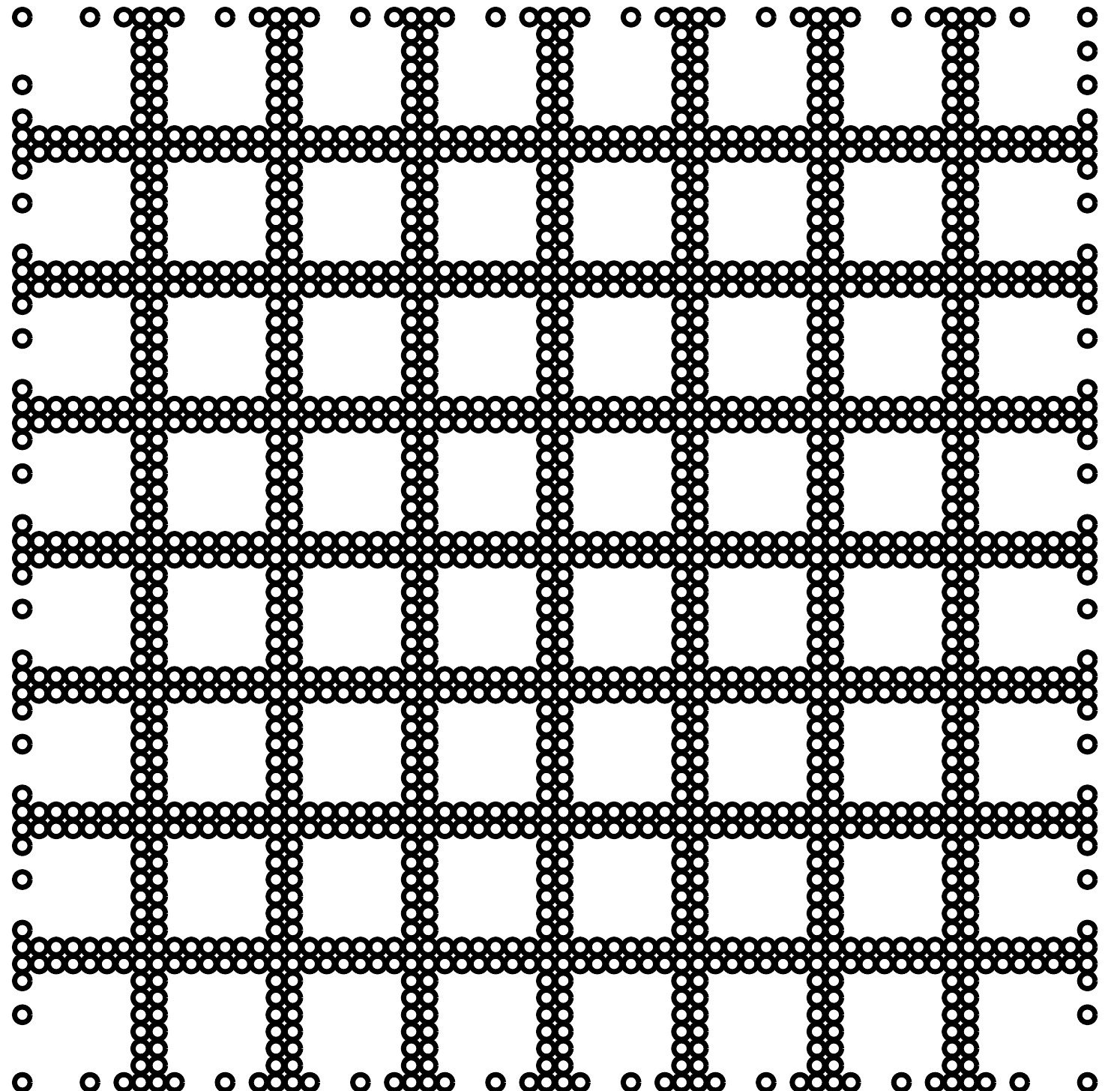}
   \caption*{$\ell = 2.5$\vspace*{-.4cm}}
  \end{subfigure}
  \qquad
  \begin{subfigure}{0.2\textwidth}
  \centering
   \includegraphics[width=0.7\textwidth]{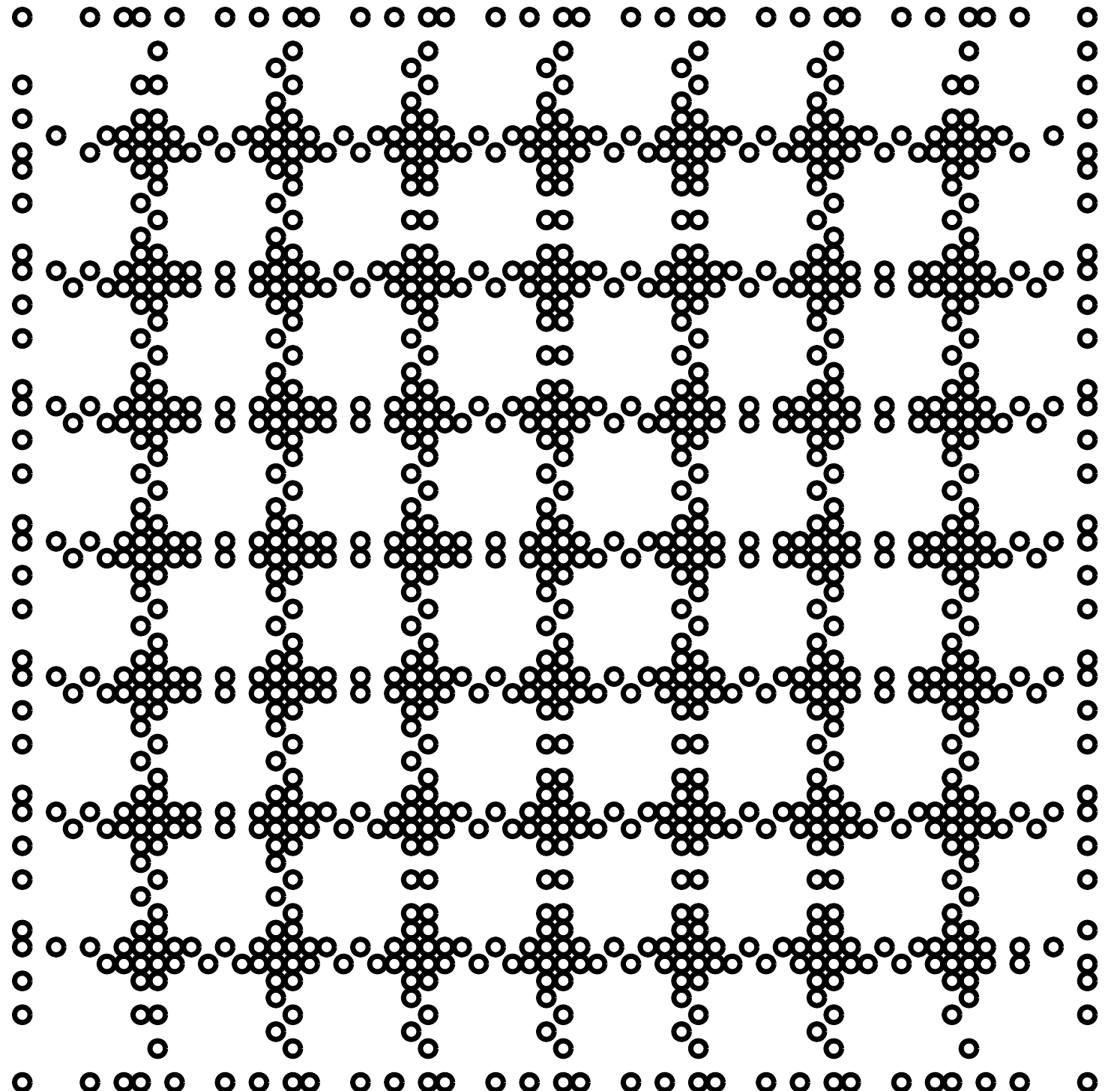}
   \caption*{$\ell = 2$\vspace*{-.4cm}}
  \end{subfigure}
  \qquad
  \begin{subfigure}{0.2\textwidth}
  \centering
   \includegraphics[width=0.7\textwidth]{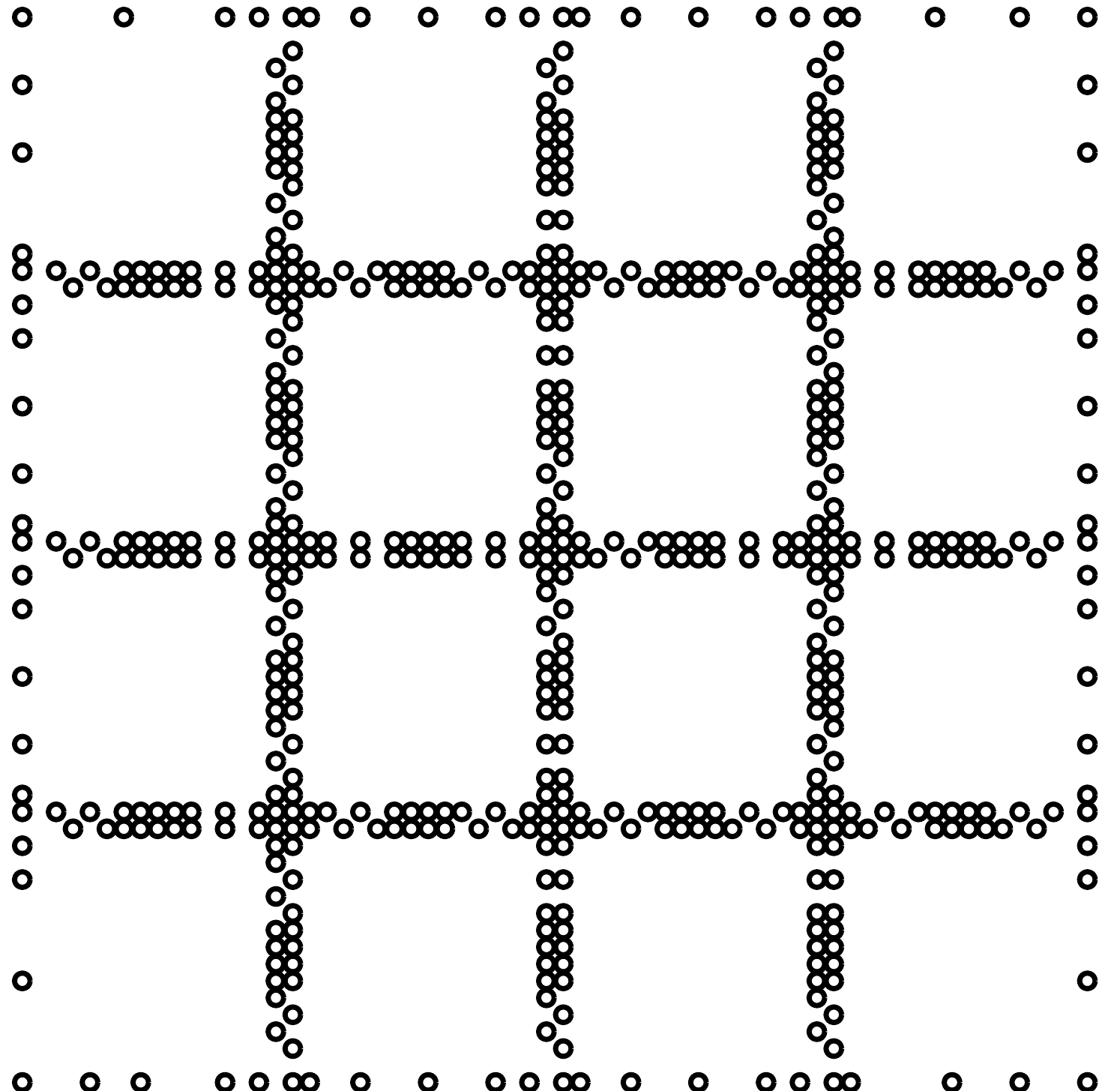}
   \caption*{$\ell = 1.5$\vspace*{-.4cm}}
  \end{subfigure}
  \\~\\~\\
  \begin{subfigure}{0.2\textwidth}
  \centering
   \includegraphics[width=0.7\textwidth]{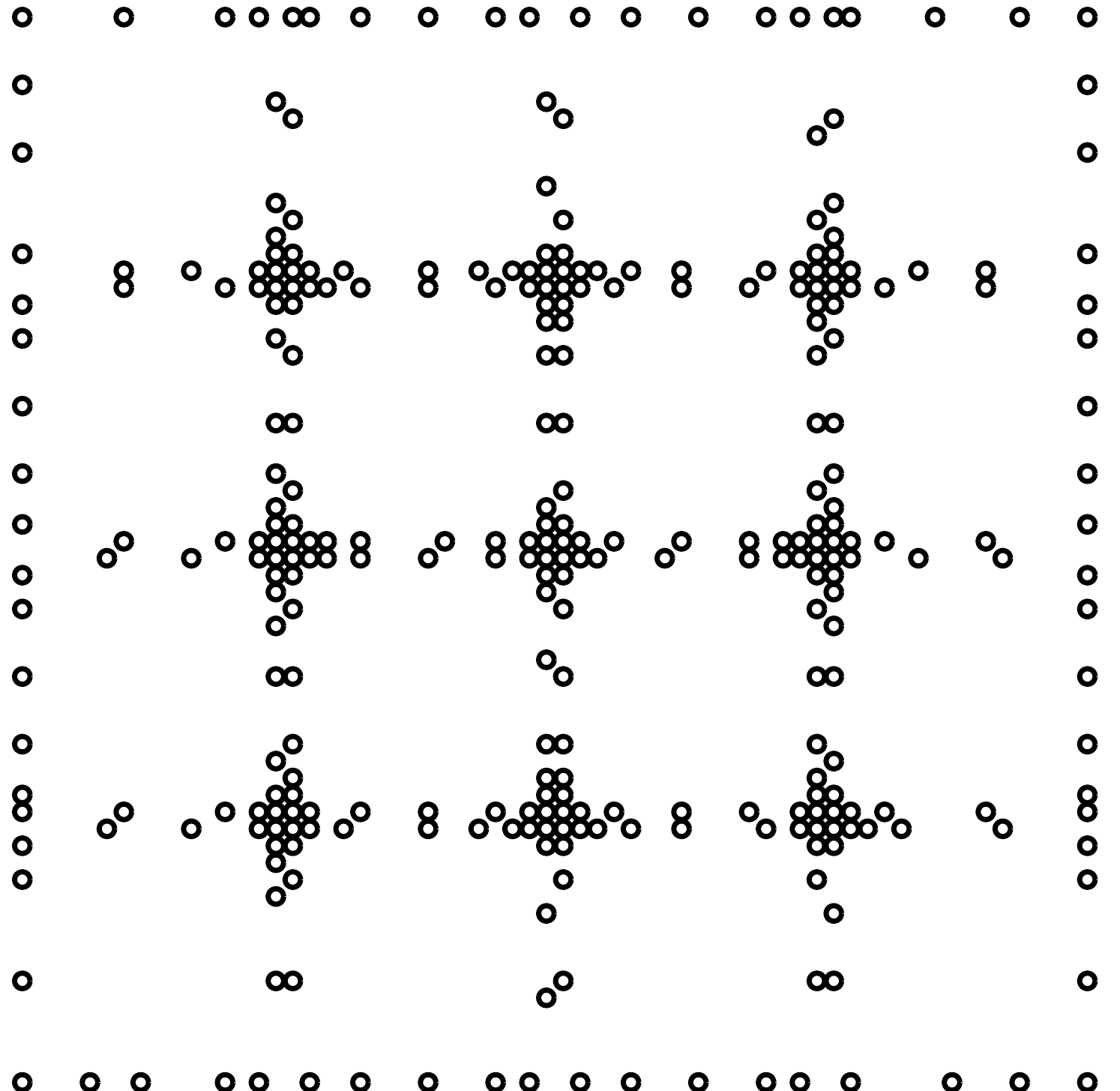}
   \caption*{$\ell = 1$\vspace*{-.4cm}}
  \end{subfigure}
  \qquad
  \begin{subfigure}{0.2\textwidth}
  \centering
   \includegraphics[width=0.7\textwidth]{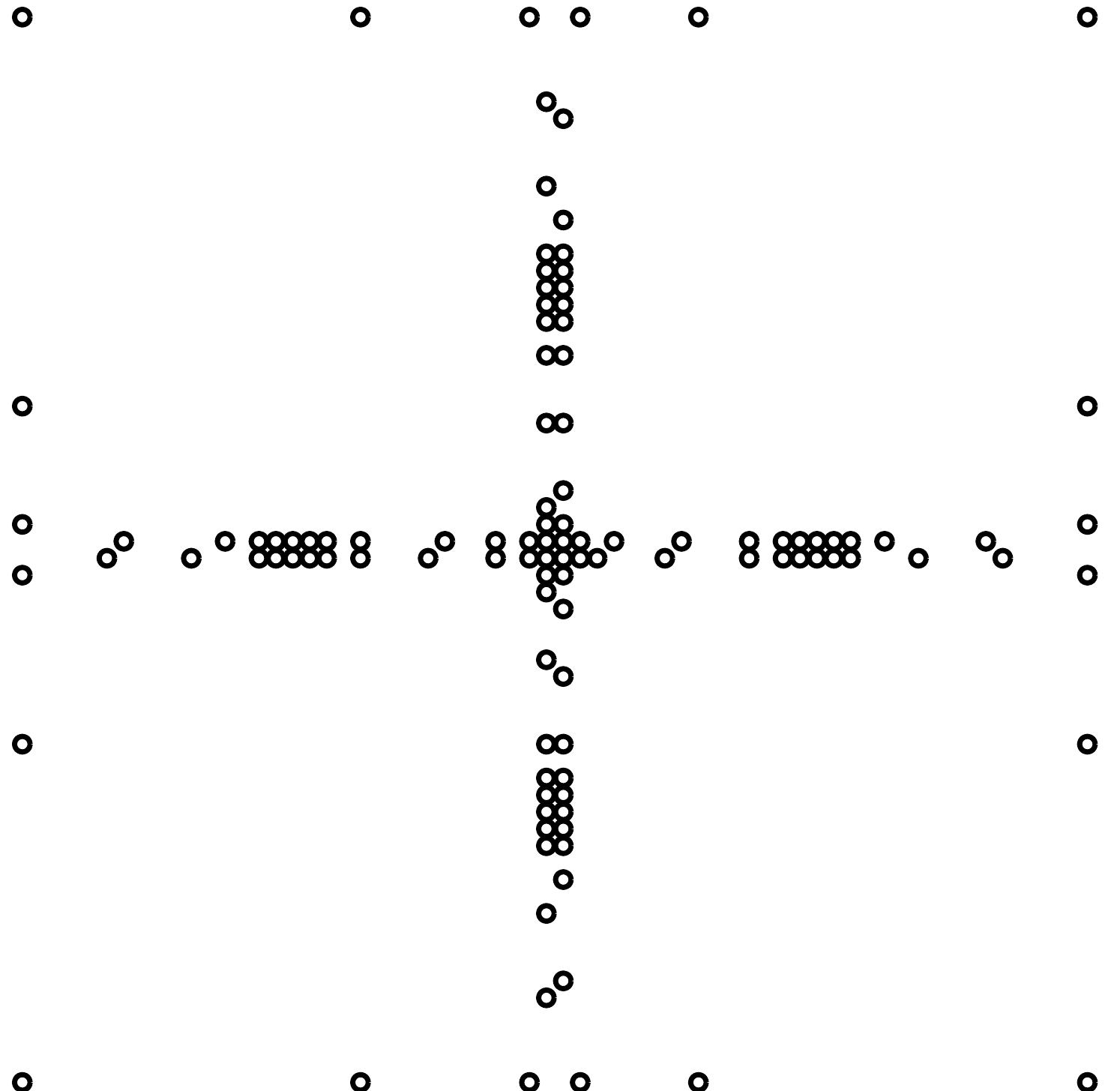}
   \caption*{$\ell = 0.5$\vspace*{-.4cm}}
  \end{subfigure}
  \qquad
  \begin{subfigure}{0.2\textwidth}
  \centering
   \includegraphics[width=0.7\textwidth]{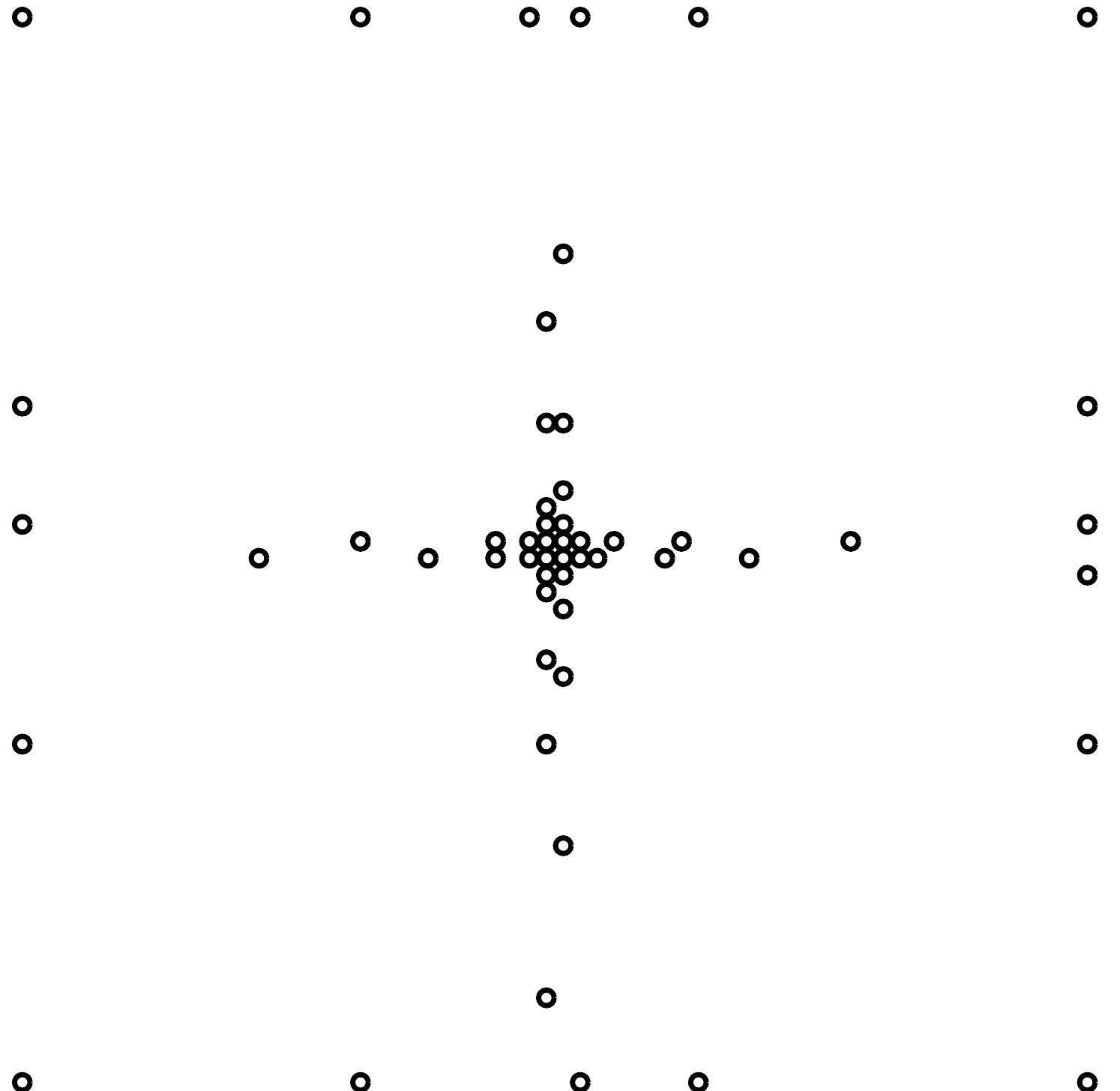}
   \caption*{$\ell = 0$\vspace*{-.4cm}}
  \end{subfigure}
  \caption{Active DOFs before skeletonizing each level $\ell$ of \texttt{hif} in 2D.  The growth of $|s_\ell|$ that was observed in the bottom of Figure \ref{fig:rskelf} appears to have been reduced dramatically. \vspace*{-.6cm}}
  \label{fig:hifie2}
 \end{figure}

There are a few idiosyncrasies to Algorithm \ref{alg:hifie} that we do not address here in detail.  For example, whereas in \rskelf{} all IDs can be shown to be applied to original blocks of the matrix $G$, in \hif{} these blocks will contain rows and columns that have been modified by Schur complement updates from the previous levels.  We direct the reader to \cite{hifie} for a thorough treatment of this and the rigorous complexity estimates.  Assuming that the edge levels admit sufficient compression (observed in practice, see Figure \ref{fig:hifie2}), however, we note that $|s_\ell|$ is asymptotically smaller than for \rskelf{} and that the computational complexity of \hif{} for elliptic kernels is now therefore $\mathcal{O}(N)$ or $\mathcal{O}(N\log N)$ for computing $F$.

\section{Updating algorithm}
\label{sec:alg}
Given the \rskelf{} and \hif{} algorithms described in Section \ref{sec:factor}, we now consider updating existing instantiations of these factorizations in response to a localized modification to the problem.  Concretely, we suppose that we have on hand a factorization corresponding to the initial problem with matrix $G$ and assume a new matrix $\new{G}$ is obtained by discretizing a locally perturbed problem as described in Section \ref{sec:localized}.  For simplicity of exposition, we initially assume that the perturbation does not modify the total number of points and does not necessitate a change in the structure of the hierarchical decomposition of space, \ie, the old quadtree is still valid for the new problem with the same occupancy bound $n_\text{occ}$, but in practice this is not necessary.  We will first discuss updating in detail for \rskelf{}, and later describe the necessary modifications for \hif{}.

As remarked in Section \ref{sec:proxy}, the use of a proxy surface as in Figure \ref{fig:proxycircle} when skeletonizing a box $b$ gives a notion of locality to the skeletonization process.  With $\N_b$ as in \eqref{eq:nbor}, $\F_b=\J_b^c\setminus\N_b$, and
\begin{align}
\left[\sk_b,\, \rd_b,\,T_{\J_b} \right] = \texttt{id}(G,\J_b,\epsilon),
\end{align}
we see that $\sk_b$, $\rd_b$, and $T_{\J_b}$ depend only on $G(\N_b,\J_b)$ and $G(\J_b, \N_b)$ and, in particular, not on $G(\F_b, \J_b)$ nor $G(\J_b, \F_b)$.  Tracing this through the rest of the skeletonization process, we therefore see that the skeletonization of $G$ with respect to box $b$ is entirely independent of $G(:, \F_b)$ and $G(\F_b,:)$, \ie, the DOF sets and matrices
\begin{align}
\left[\sk_b, \,\rd_b,\, D_{\sk_b,\sk_b},\, D_{\rd_b,\rd_b},\, U_{\J_b},\, V_{\J_b} \right] = \texttt{skel}(G,\J_b,\epsilon)
\end{align}
can be computed without looking at those entries of $G$.

Based on the above observation, it is not difficult to see that if $b$ is a box at level $L$ and $\new{G}(\F_b^c,\F_b^c) = G(\F_b^c,\F_b^c)$, then
$\texttt{skel}(\new{G},\J_b,\epsilon) = \texttt{skel}(G,\J_b,\epsilon)$ and the corresponding blocks and index sets do not need to be recomputed.  Similar logic can be applied at quadtree levels $\ell < L$ by taking into account some propagation rules we outline below.

\subsection{Propagation rules}\label{sec:prop}
We begin by defining the collection of DOF sets of boxes $b$ on level $L$ for which it is possibly the case that $\texttt{skel}(\new{G},\J_b,\epsilon) \ne \texttt{skel}(G,\J_b,\epsilon)$, which from our previous discussion is the collection
\begin{align}
\M_L = \left\{\J_b \, \vert \, \text{ $b$ is a box on level $L$ and } \new{G}(\new{\F}_b^c,\new{\F}_b^c) \ne G(\F_b^c,\F_b^c) \right\}.
\end{align}
We refer to this as the collection of \emph{marked} DOF sets (or simply marked boxes) on level $L$, and the remainder of this section is dedicated to describing the rules that determine for which boxes at levels $\ell < L$ the output of $\texttt{skel}(\new{G},\J_b,\epsilon)$ may differ and therefore must also be marked at the appropriate level.  We will use an overbar (\eg, $\new{\sk}_{\J_b}$) to distinguish between quantities corresponding to the new factorization of $\new{G}$ and the old factorization of $G$ when necessary.

It is a simple consquence of Algorithm \ref{alg:rskelf}, that in \rskelf{} the diagonal blocks satisfy a nesting property.  By this we mean that, if $b$ is a box on level $\ell < L$ with child boxes $\child(b)$, then for each $b'\in\child(b)$ it is the case that $\new{D}_{\new{\sk}_{b'},\new{\sk}_{b'}}$ is a subblock of $\new{G}_\ell(\J_b,\J_b)$.  This leads to perhaps the most self-evident propagation rule: if box $b$ is marked, then so is the parent box of $b$, $\parent(b)$.  Based on this, we define the collection
\begin{align}
\P_\ell &= \left\{\J_b \, \vert \, b= \parent(b') \text{ for some } b' \text { with } \J_{b'}\in\M_{\ell+1}  \right\}
\end{align}
for $\ell =1,\dots,L-1$

Beyond the simple child-to-parent rule, we assert that, if a node $b$ on level $\ell+1$ is marked then every $b'$ on level $\ell$ such that $\parent(b)\in\nbor(b')$ is also marked.  This is because the set $\new{\N}_{b'}\cap \new{\sk}_{b}$ will be non-empty, and, since $b$ is marked, it is thus possible that blocks involved in the ID with respect to $\J_{b'}$ will have changed.  Thus, for each $\ell < L$ we define
\begin{align}
\U_\ell &= \left\{\J_{b'} \, \vert \, b\in\nbor(b') \text{ for some } b \text { with } \J_b\in\P_\ell  \right\}.
\end{align}

Finally we note that, due to heterogeneous refinement, it is possible that there are leaf boxes $b$ at levels $\ell < L$ that have been directly modified, \ie, $\new{G}(\new{\F}_b^c,\new{\F}_b^c) \ne G(\F_b^c,\F_b^c)$.  Such boxes are also clearly marked, though they may not be covered by the previous two rules.  Combining this rule with the previous two leads us to define the collection of marked DOF sets for levels $\ell < L$ as
\begin{align}
\M_\ell = \left\{\J_b \, \vert \, \text{ $b$ is a box on level $\ell$ and } \new{G}(\F_b^c,\F_b^c) \ne G(\F_b^c,\F_b^c) \right\}\cup\P_\ell\cup\U_\ell.
\end{align}
We see an example of the evolution of the marked set $\M_\ell$ in Figure \ref{fig:hi}.

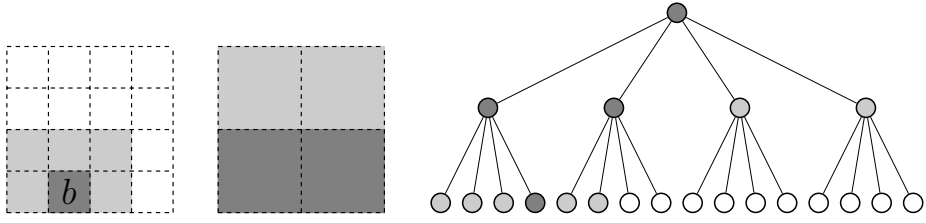
\begin{figure}
\centering

\scalebox{0.85}{
\scalebox{0.65}{
  \begin{tikzpicture}

  \filldraw[lightgray] (0,0) rectangle (3, 2);

\filldraw[darkgray,thick] (1,0) rectangle (2,1);
  \draw[step=1cm,black,dashed,thick] (0,0) grid (4,4);

  \node[font=\fontsize{44}{44}] at (1.5,0.5) {$b$};
  \end{tikzpicture}
  }
\hspace{0.3cm}
\scalebox{0.65}{
\begin{tikzpicture}

  \filldraw[lightgray] (0,2) rectangle (4, 4);
  \filldraw[darkgray,thick] (0,0) rectangle (4,2);
  \draw[step=2cm,black,dashed,thick] (0,0) grid (4,4);
\end{tikzpicture}
}
\hspace{0.3cm}
\scalebox{0.85}{
\begin{tikzpicture}
\Tree
[.\node[thick, fill=darkgray]{};
    [.\node[thick, fill=darkgray]{};
        \node[thick,fill=lightgray]{};
        \node[thick,fill=lightgray]{};
        \node[thick,fill=lightgray]{};
        \node[thick,fill=darkgray]{};
    ]
    [.\node[thick, fill=darkgray]{};
        \node[thick, fill=lightgray]{};
        \node[thick, fill=lightgray]{};
        \node[thick]{};
        \node[thick]{};
    ]
    [.\node[thick,fill=lightgray]{};
        \node[thick]{};
        \node[thick]{};
        \node[thick]{};
        \node[thick]{};
    ]
    [.\node[thick,fill=lightgray]{};
        \node[thick]{};
        \node[thick]{};
        \node[thick]{};
        \node[thick]{};
    ]
]
\end{tikzpicture}
}
}
\caption{\label{fig:hi}
 Left: Suppose the local perturbations are contained in box $b$ so that $\new{G}(:, \J_b) \neq G(:,\J_b)$ and $\new{G}(\J_b, :) \neq G(\J_b, :)$.  Initially $\M_L$ contains the DOF sets corresponding to the shaded boxes.  Center: At level $L-1$, DOF sets corresponding to the dark gray boxes are in $\P_{L-1}$ and thus $\M_{L-1}$
  because they have marked children, and the light
  gray boxes are in $\U_{L-1}$ and thus $\M_{L-1}$ because they have neighbors in $\P_{L-1}$  Right: The corresponding
  quadtree with nodes shaded the same as their associated boxes.
  \vspace*{-.6cm}}
\end{figure}

\subsection{Updating a group skeletonization}\label{sec:reskel}
At any level $\ell$ of the \rskelf{} algorithm, we have remarked that in the group skeletonization of the new system
\begin{align}
\left[\left\{\new{\sk}_b,\, \new{\rd}_b,\, \new{D}_{\new{\sk}_b,\new{\sk}_b},\, \new{D}_{\new{\rd}_b,\new{\rd}_b}\right\}_{\J_b \in \L_\ell},\,\new{U}_{\L_\ell},\, \new{V}_{\L_\ell}\right] = \texttt{skel}(\new{G}_\ell,\L_\ell,\epsilon)
\end{align}
there will be a collection of DOF sets and associated boxes for which the corresponding blocks and index sets output above do not differ from those computed in the factorization of $G$, namely, the collection $\L_\ell \setminus \M_\ell$.  For those boxes $b\in\M_\ell$ for which the corresponding output does potentially differ, we see that in computing $\texttt{skel}(\new{G}_\ell,\J_b,\epsilon)$ we require knowledge of $\new{G}_\ell(\J_b,\J_b)$.  As we have commented in the previous section, it is the case that $\new{G}_\ell(\new{\sk}_{b'},\new{\sk}_{b'}) = \new{D}_{\new{\sk}_{b'},\new{\sk}_{b'}}$ for each $b' \in \child(b)$.  In \rskelf{}, however, it is also the case that all other blocks of $\new{G}_\ell$ that are relevant to skeletonization with respect to $\J_b$ are pure kernel interactions, \ie, they coincide with blocks of $\new{G}$.

Notationally, when writing the group skeletonization $\skel_{\L_\ell}(\new{G}_\ell)$ as in \eqref{eq:groupskel}, we can partition the matrices $\new{U}_{\L_\ell}$ and $\new{V}_{\L_\ell}$ into separate products over $\M_\ell$ and $\L_\ell\setminus \M_\ell$ as
\begin{align}
\skel_{\L_\ell}(\new{G}_\ell) &\approx \new{U}_{\L_\ell}^*\new{G}_\ell \new{V}_{\L_\ell} = \new{U}_{\M_\ell}^*U_{\L_\ell\setminus\M_\ell}^*\new{G}_\ell V_{\L_\ell\setminus\M_\ell}\new{V}_{\M_\ell}\\
 &= \left(\prod_{\J\in \M_\ell}\new{U}_\J^*\right)\left(\prod_{\J\in \L_\ell\setminus\M_\ell}U_\J^*\right)\new{G}_\ell\left(\prod_{\J\in \L_\ell\setminus\M_\ell}V_\J\right)\left(\prod_{\J\in \M_\ell}\new{V}_\J\right),
\end{align}
where the matrices $U_{\L_\ell\setminus\M_\ell}$ and $V_{\L_\ell\setminus\M_\ell}$ are factors of the original $U_{\L_\ell}$ and $V_{\L_\ell}$ matrices from the factorization of $G_\ell$.  With this, we can functionally write the necessary computation to update this skeletonization as
\begin{align}
\left[\left\{\new{\sk}_b,\, \new{\rd}_b,\, \new{D}_{\new{\sk}_b,\new{\sk}_b},\, \new{D}_{\new{\rd}_b,\new{\rd}_b}\right\}_{\J_b \in \M_\ell},\,\new{U}_{\M_\ell},\, \new{V}_{\M_\ell} \right] = \texttt{skel\_update}(\new{G}_\ell,\L_\ell,\M_\ell,\epsilon),
\end{align}
explicitly avoiding redundant recomputation of the blocks we already know.

\subsection{Updating \rskelf{}}\label{sec:update}
Given the previous discussion, updating becomes a simple
two-step process for each level $\ell=L,\dots,1$.  First, the
propagation rules must be applied to determine the marked set
$\M_\ell$ for the current level.  Then, the group skeletonization is
updated according to the process outlined in Section
\ref{sec:reskel}.  Repeating this level-by-level, we obtain Algorithm
\ref{alg:update}, which is intentionally written analogously to
Algorithm \ref{alg:rskelf}.  It is important to note that the updating
process here is not an approximate one: \emph{the updated
  factorization $F$ that is obtained is identical (to machine precision) to that which would
  have been obtained starting from scratch with the same decomposition of space}.  In other words, while \rskelf{} is accurate to specified tolerance $\epsilon$ by design, no additional approximation error is introduced
in updating the factorization and there is no compounding of error
with repeated updates.

Thus far, we have considered updates that do not change the structure
of the hierarchical decomposition.  In the case of tree refinement
where quadtree nodes are created or deleted the core updating
algorithm does not change, but the collection $\L_\ell$ will now
itself have changed, potentially containing more or fewer DOF sets, and it is necessary to
perform some minor bookkeeping to ensure that DOF sets corresponding to new boxes are always
in the marked set and that factors corresponding to deleted nodes are
removed from the factorization.  In cases where the structure of the hierarchical decomposition changes, it is necessary to use a fully adaptive data structure for the quadtree such that the addition or removal of points causes corresponding refinement or coarsening of the decomposition to obtain the exact factorization via updating as one would have from factoring anew.  Without such dynamic tree maintenance, the updated factorization will still be accurate to the specified tolerance but will not be numerically the same factorization.  In our examples, we construct updates such that the same underlying tree structure is obtained and updating is numerically exact.

\begin{algorithm}
\small
  \caption{Updating \rskelf{}}
  \label{alg:update}
  \begin{algorithmic}
  \State $\new{G}_L = \new{G}$
   \For{$\ell = L, L-1, \dots, 1$}
   \State{\texttt{// get updated skeleton blocks and operators}}
   \State $\left[\left\{\new{\sk}_b,\, \new{\rd}_b,\, \new{D}_{\new{\sk}_b,\new{\sk}_b},\, \new{D}_{\new{\rd}_b,\new{\rd}_b}\right\}_{\J_b \in \M_\ell},\,\new{U}_{\M_\ell},\, \new{V}_{\M_\ell} \right] = \texttt{skel\_update}(\new{G}_\ell,\L_\ell,\M_\ell,\epsilon)$
   \State{\texttt{// assemble skeletonization}}
   \State $\new{G}_{\ell-1} = \new{G}_{\ell}$
   \For{$\J_b \in \L_\ell\setminus\M_\ell$ with $\J_b=\sk_b\cup\rd_b$}
    \State $\new{G}_{\ell-1}(:,\rd_b) = G_{\ell-1}(\rd_b,:) = 0$
    \State $\new{G}_{\ell-1}(\sk_b,\sk_b) = D_{\sk_b,\sk_b}$
    \State $\new{G}_{\ell-1}(\rd_b,\rd_b) = D_{\rd_b,\rd_b}$
    \EndFor
    \For{$\J_b \in \M_\ell$ with $\J_b=\new{\sk}_b\cup\new{\rd}_b$}
    \State $\new{G}_{\ell-1}(:,\new{\rd}_b) = \new{G}_{\ell-1}(\new{\rd}_b,:) = 0$
    \State $\new{G}_{\ell-1}(\new{\sk}_b,\new{\sk}_b) = \new{D}_{\new{\sk}_b,\new{\sk}_b}$
    \State $\new{G}_{\ell-1}(\new{\rd}_b,\new{\rd}_b) = \new{D}_{\new{\rd}_b,\new{\rd}_b}$
    \EndFor
   \EndFor
   \State $\new{G} \approx \new{F} \equiv \new{U}_{\M_L}^{-*}U_{\L_L\setminus \M_L}^{-*} \cdots \new{U}_{\M_1}^{-*} U_{\L_1\setminus\M_1}^{-*} \new{G}_0 V_{\L_1\setminus \M_1}^{-1}\new{V}_{\M_1}^{-1} \cdots V_{\L_L\setminus\M_L}^{-1}\new{V}_{\M_L}^{-1}$
  \end{algorithmic}
 \end{algorithm}
 Again, the actual assembly of $\bar G_{\ell-1}$ is not necessary and is purely notational.

\subsection{Complexity of updating \rskelf{}}
Intuitively, if a perturbation between $G$ and $\new{G}$ is localized, then the total number of marked boxes
\begin{align*}
\left|\M_0\right| + \left|\M_1\right| + \dots + \left|\M_L\right|
\end{align*}
(\ie, the number of shaded nodes in Figure \ref{fig:hi}) is
asymptotically smaller than the total number of boxes in the hierarchy
in a way that will be made rigorous.  Thus, if each box has roughly
the same skeletonization cost, we see that updating will be
asymptotically less expensive than performing a new refactorization from scratch.

The assumptions that ensure \rskelf{} is computationally
efficient (asymptotic complexity $\mathcal{O}(N)$ to factor)
serve to control the number of active DOFs at each level.  Let
$k_\ell$ independent of $b$ denote a bound on the number of skeleton DOFs $|\new{\sk}_b|$ or $|\sk_b|$ for a box $b$ at level $\ell$.  For \rskelf{}, standard multipole estimates
show that factorization of an elliptic system with a quasi-1D boundary
leads to $k_\ell$ growing linearly as we progress up the tree, \ie,
$k_\ell =\mathcal{O}(L-\ell)$ (\emph{cf.}
\cite{martinsson-rokhlin,hifie}).  The cost of skeletonizing a box is
dominated by the cost of the ID, which is cubic in the number of
skeleton DOFs. Therefore the cost of skeletonizing a box on level $\ell$ is $\mathcal{O}\left((L-\ell)^3\right)$.

After a single leaf-level perturbation, \ie, a perturbation that is localized such that $\new{G}(\J_b,:) \ne G(\J_b,:)$ and  $\new{G}(:, \J_b) \ne G(:,\J_b)$ only for a single leaf box $b$, Lemma \ref{lem:lem} shows that
the total number of boxes that need to have their components of the skeletonization updated at any level $\ell$ is bounded
by a small constant $C$.  With this lemma, we show in the proof of
Theorem \ref{th:rskelupdate} that the cost of updating after $m$
leaf-level perturbations has asymptotic complexity
$\mathcal{O}(m\log^4 N)$, \ie, linear in $m$ but only poly-log in the
total number of DOFs.

\begin{lemma}\label{lem:lem}
  Suppose that $\new{G}(\J_{b_L},:) \ne G(\J_{b_L},:)$ and  $\new{G}(:, \J_{b_L}) \ne G(:,\J_{b_L})$ only for a single leaf box $b_L$ on level $L$.  Then the size of the marked set, $|\M_\ell|$, is bounded by a small dimension-dependent constant $C$ independent
  of $N$ and $\ell$.
\end{lemma}
\begin{proof}
  In any dimension, $d$, we can associate each box $b$ on a given
  level with a $d$-tuple of integer coordinates, $(z_1,\dots,z_d),$
  corresponding to the center of the box in the grid at that level.
  It is natural to consider the $\ell_\infty$-distance associated with
  this representation,
\begin{align*}
\|b - b'\|_\infty \equiv \|\left(z_1,\,\dots\,,z_d\right)  - \left(z'_1,\,\dots,\,z_d'\right)\|_\infty,
\end{align*}
which codifies the idea ``$b$ is $\|b - b'\|_\infty$ boxes away from
$b'$''.  With this distance in mind, we begin by defining the concept
of \emph{reach} at a level, $r_\ell$.  With $b_\ell$ the single ancestor of $b_L$ at level $\ell$, we define the reach at level $\ell$ as
\begin{align*}
r_\ell \equiv \max_{b\in \M_\ell} \|b_\ell - b \|_\infty,
\end{align*} \ie, it is the $\ell_\infty$ radius of the marked set at level $\ell$.

The key observation is that
the bottom level reach is $r_L\equiv1$ and the reach at subsequent levels
does not much exceed this size.  In particular, the reach satisfies
the recurrence relation
\begin{align*}
r_\ell = \left\lceil\frac{r_{\ell+1}}{2}\right\rceil + 1,
\end{align*}
where division by two corresponds to the fact that marked boxes on a
level are contiguous and $r_{\ell+1}$ contiguous boxes have at most
$\left\lceil{\frac{r_{\ell+1}}{2}}\right\rceil$ parents, and adding
one corresponds to marking all neigbors of these parents. This
relation has a fixed point at $r_\ell=2$.  Therefore, in $d$
dimensions, the size of the marked set is bounded as
$\lvert\M_\ell\rvert \le (2r_\ell + 1)^d \le C$.
\end{proof}

\begin{theorem}[Complexity of updating \rskelf{}]\label{th:rskelupdate}Assume that the the number of
  skeletons for a box at level $\ell$, $k_\ell$, grows like
  $\mathcal{O}(L-\ell)$.
Suppose we use the updating technique of Section \ref{sec:update} to construct an updated factorization of $\new{G}$ given a factorization of $G$, where $\new{G}(\J_{b},:) \ne G(\J_{b},:)$ and  $\new{G}(:, \J_{b}) \ne G(:,\J_{b})$ only for a collection of boxes $b$ of size $m$.
Then, for an integral equation with elliptic kernel on a quasi-1D domain, the complexity of updating \rskelf{} is
$\mathcal{O}(m\log^4 N)$.
\end{theorem}
\begin{proof}
On level $\ell$ we need to update the skeletonization blocks corresponding to $\lvert \M_\ell \rvert$ boxes, each of which has at most $k_\ell$ skeleton DOFs.  This costs $\mathcal{O}(k_\ell^3)$ per box.  This means the total re-factorization time, $t$, grows as
\begin{align*}
t = \sum_{\ell=0}^L \lvert \M_\ell \rvert \mathcal{O}(k_\ell^3) = \mathcal{O}\left(mC\sum_{\ell=0}^L (L-\ell)^3\right) = \mathcal{O}(mC\log^4 N),
\end{align*}
where we have used the fact that the marked set resulting from $m$
leaf-level modifications is no bigger than $m$ times the maximum
marked set size of a single box, $C$ from Lemma \ref{lem:lem}, as well as the fact that our quadtree is constructed such that $L=\mathcal{O}(\log N)$.  Note
that this bound is clearly weak, as the number of marked boxes
on a level is of course limited by the total number of boxes on that
level; for example, if $m = \mathcal{O}(N)$ then $t=\mathcal{O}(N)$.
\end{proof}

\subsection{Modifications for \hif{}}
To adapt the updating process for \rskelf{} to an updating
process for \hif{}, we use the same basic building blocks of
identifying marked boxes (and now marked edges) and updating the corresponding skeletonizations.
Incorporating the half-integer edge levels, however, complicates the
process.

In Section \ref{sec:reskel} we used the nesting
property of diagonal blocks in \rskelf{} to assert that, for a box $b$
$\new{G}_\ell(\new{\sk}_{b'},\new{\sk}_{b'}) = \new{D}_{\new{\sk}_{b'},\new{\sk}_{b'}}$ for each $b' \in \child(b)$ and all other entries of $\new{G}_\ell(\J_b,\J_b)$ are pure kernel interactions.  For \hif{}, this is no longer the case due to mixing of matrix blocks between box and edge levels, and thus keeping track of the
state of interactions between DOFs becomes more complicated.

In particular, in \hif{}, the block $\new{G}_\ell(\J_b,\J_b)$ for a box $b$ at level $\ell < L$ has only a subset of entries that come directly from the blocks $\new{D}_{\new{\sk}_{b'},\new{\sk}_{b'}}$ corresponding to $b'\in\child(b)$.  Other entries have now recieved Schur complement updates from the skeletonization of edges at level $\ell+1/2$, see
Figure \ref{fig:edgeupdates}.  The intuition for defining $\M_\ell$ and $\M_{\ell-1/2}$ follows essentially the same reasoning as for \rskelf{}, taking into account this extra mixing of information due to Schur complement updates. Because of this, the marked set $\M_\ell$ for \hif{}
will be larger than that for \rskelf{}, though not asymptotically
so.

Just as in the \rskelf{} case, the updating procedure described here for \hif{} is exact in that the same factorization is obtained as would have been obtained when computing a new \hif{} factorization on the same decomposition of space.

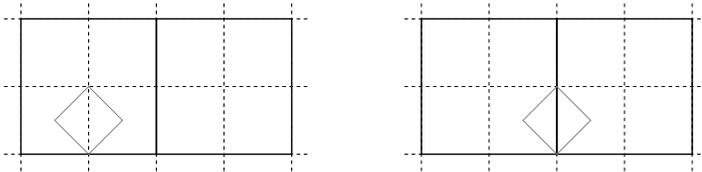
\begin{figure}
\centering
\begin{subfigure}{0.4\textwidth}
  \scalebox{0.45}{
      \begin{tikzpicture}
      \draw[step=2cm,black,dashed,thick] (-0.5,-0.5) grid (8.5,4.5);
      \draw[black,very thick] (0,0) rectangle (4,4);
      \draw[black,very thick] (4,0) rectangle (8,4);
      \draw[thick,darkgray,shift={(2,0)},rotate=45] (0,0) rectangle (1.414,1.414);
 \end{tikzpicture}
 }
\end{subfigure}
\begin{subfigure}{0.4\textwidth}
  \scalebox{0.45}{
      \begin{tikzpicture}
      \draw[step=2cm,black,dashed,thick] (-0.5,-0.5) grid (8.5,4.5);
      \draw[black,very thick] (0,0) rectangle (4,4);
      \draw[black,very thick] (4,0) rectangle (8,4);
      \draw[thick,darkgray,shift={(4,0)},rotate=45] (0,0) rectangle (1.414,1.414);
 \end{tikzpicture}
 }
\end{subfigure}
\caption{\label{fig:edgeupdates} Left: When we skeletonize with respect to the DOFs associated with the shown Voronoi cell of an edge between two small boxes $b_1$ and $b_2$ with the same parent $b$ at level $\ell$, a Schur complement update is performed that modifies the entries $\new{G}_\ell(\new{\sk}_{b_1},\new{\sk}_{b_2})$ and $\new{G}_\ell(\new{\sk}_{b_2},\new{\sk_{b_1}})$ to no longer be original subblocks of $\new{G}$, which modifies the block $\new{G}_\ell(\J_b,\J_b)$.  Right: Similarly, when skeletonizing with respect to the DOFs associated with the Voronoi cell of an edge that also comprises part of an edge of $b$, a Schur complement update occurs that modifies the blocks $\new{G}(:,\J_b)$ and $\new{G}(\J_b,:)$, entries of which are used in the ID portion of skeletonization with respect to $\J_b$. \vspace*{-.6cm}}
\end{figure}

\subsection{Complexity of updating \hif{}}

The asymptotic complexity of updating \hif{} using the same
technique as for \rskelf{} follows essentially the same path of
reasoning.
\begin{theorem}[Complexity of updating \hif{}]\label{th:hifupdate} Assume that the the number of
  skeletons for a box at level $\ell$, $k_\ell$, grows like
  $\mathcal{O}(L-\ell)$.   Suppose we use the updating technique of Section \ref{sec:update} to construct an updated factorization of $\new{G}$ given a factorization of $G$, where $\new{G}(\J_{b},:) \ne G(\J_{b},:)$ and  $\new{G}(:, \J_{b}:) \ne G(:,\J_{b})$ only for a collection of boxes $b$ of size $m$.
Then, for an integral equation with elliptic kernel on a 2D domain, the complexity of updating \hif{} is
$\mathcal{O}(m\log^4 N)$.
\end{theorem}
\begin{proof}
The proof is essentially the same as that of Theorem
\ref{th:rskelupdate} taking into account edge-level skeletonization as
well as box-level skeletonization in a manner analogous to Lemma
\ref{lem:lem}, yielding larger constants.  Writing the recurrence relation (at the box level) for the reach in \hif{} as we did for \rskelf{}, we obtain
\begin{align*}
r_\ell = \left\lceil \frac{r_{\ell + 1} + 2}{2} \right\rceil + 1 = \left\lceil \frac{r_{\ell + 1}}{2} \right\rceil + 2,
\end{align*}
which has a fixed point at 4 (versus 2 for \rskelf{}). The same trick as before can be used to bound the size of the collection of marked boxes, giving
$\lvert\M_\ell\rvert \le (2r_\ell + 1)^d \le C'$.
\end{proof}

In Theorem \ref{th:hifupdate} we assume that the number of remaining
skeletons for a box at level $\ell$ grows like $\mathcal{O}(L-\ell)$.
This rate of growth is strongly supported by numerical experiments
(see \cite{hifie}) though remains a conjecture at this time. For both
\hif{} and \rskelf{}, it is interesting to observe that
when a constant number of leaf boxes are modified, the cost of an
update is asymptotically less expensive than an apply or solve, both of which have complexity $\mathcal{O}(N)$.

\section{Numerical results}
\label{sec:results}
 We now present two examples showing the asymptotic scaling of our updating routine, one for \rskelf{} and one for \hif{}.  For each example, the following, if applicable, are given:
 \begin{itemize}
  \item
   $\epsilon$: base relative precision of the interpolative decomposition;
  \item
   $N$: total number of DOFs in the problem;
  \item
   $t_{\text{f}}$: wall clock time for constructing the factorization in seconds;
   \item
   $t_{\text{u,p}}$: wall clock time for updating in response to modifying a constant proportion of points in the factorization;
  \item
   $t_{\text{u,n}}$: wall clock time for updating in response to modifying a constant number of points in the factorization.
 \end{itemize}
 All algorithms and examples were implemented in C++ using the Intel Math Kernel Library for BLAS/LAPACK routines, and all computations were performed using a single core of an Intel Xeon E5-4640 CPU at 2.4 GHz on a 64-bit Linux machine with 1.5 TB of RAM.  Previous work in \cite{domainsAd,rskel,hifie,martinsson-rokhlin} has shown that the accuracy of the approximate factorization is well-controlled by $\epsilon$, in that the \rskelf{} or \hif{} factorization $F$ of a matrix $G$ satisfies
 \begin{align}
\|G - F\|_2 \lesssim \epsilon \|G\|_2.
 \end{align} As such, we focus our discussion on the asymptotic runtime of factoring versus updating.

\subsection{Example 1: Laplace double-layer potential on a circle with a bump}
We first present an example of modifying the boundary geometry for a boundary integral equation formulation of the Laplace equation.  Consider the interior Dirichlet Laplace problem,
\begin{align}
\Delta u(x) &= 0, & &x\in\Omega\subset\mathbb{R}^2,\\
u(x) &=f(x), & &x\in\Gamma=\partial\Omega,
\end{align} which can be written as a second-kind integral equation with unknown surface density $\sigma(x)$ as
\begin{align}\label{eq:bielap}
-\frac{1}{2}\sigma(x) + \int_\Gamma \frac{\partial K}{\partial\nu_y}(\|x-y\|)\sigma(y)\,d\Gamma(y) &=f(x), & &x\in\Gamma,
\end{align}
where $K(r) = -\frac{1}{2\pi}\log r$ is the fundamental solution of the free-space partial differential equation and $\nu_y$ is the outward-facing unit normal at $y\in\Gamma$.

\begin{figure}
\centering
\begin{subfigure}{0.4\textwidth}
\includegraphics[width=0.9\textwidth]{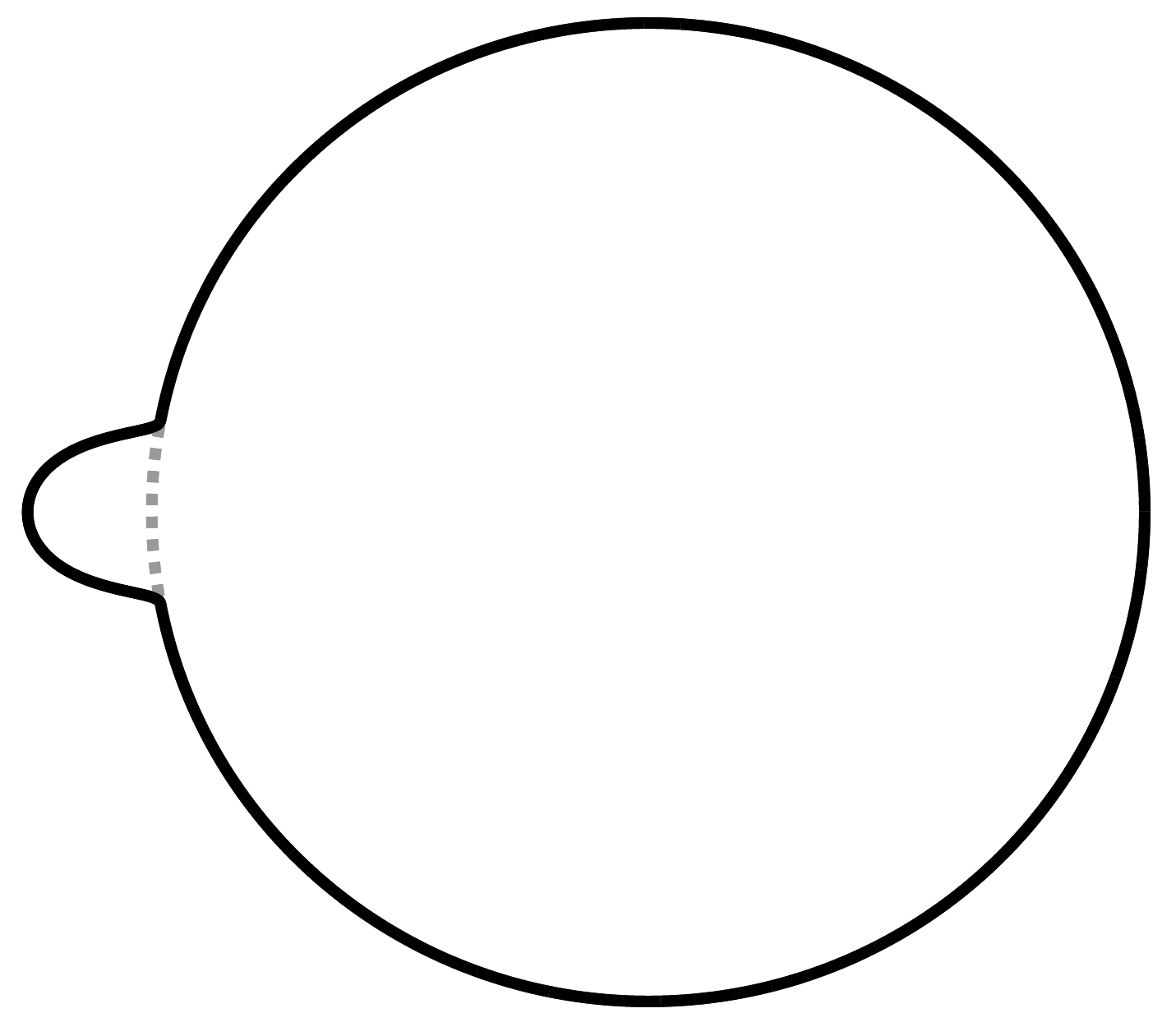}
\end{subfigure}
\qquad
\begin{subfigure}{0.4\textwidth}
\includegraphics[width=1.0\textwidth]{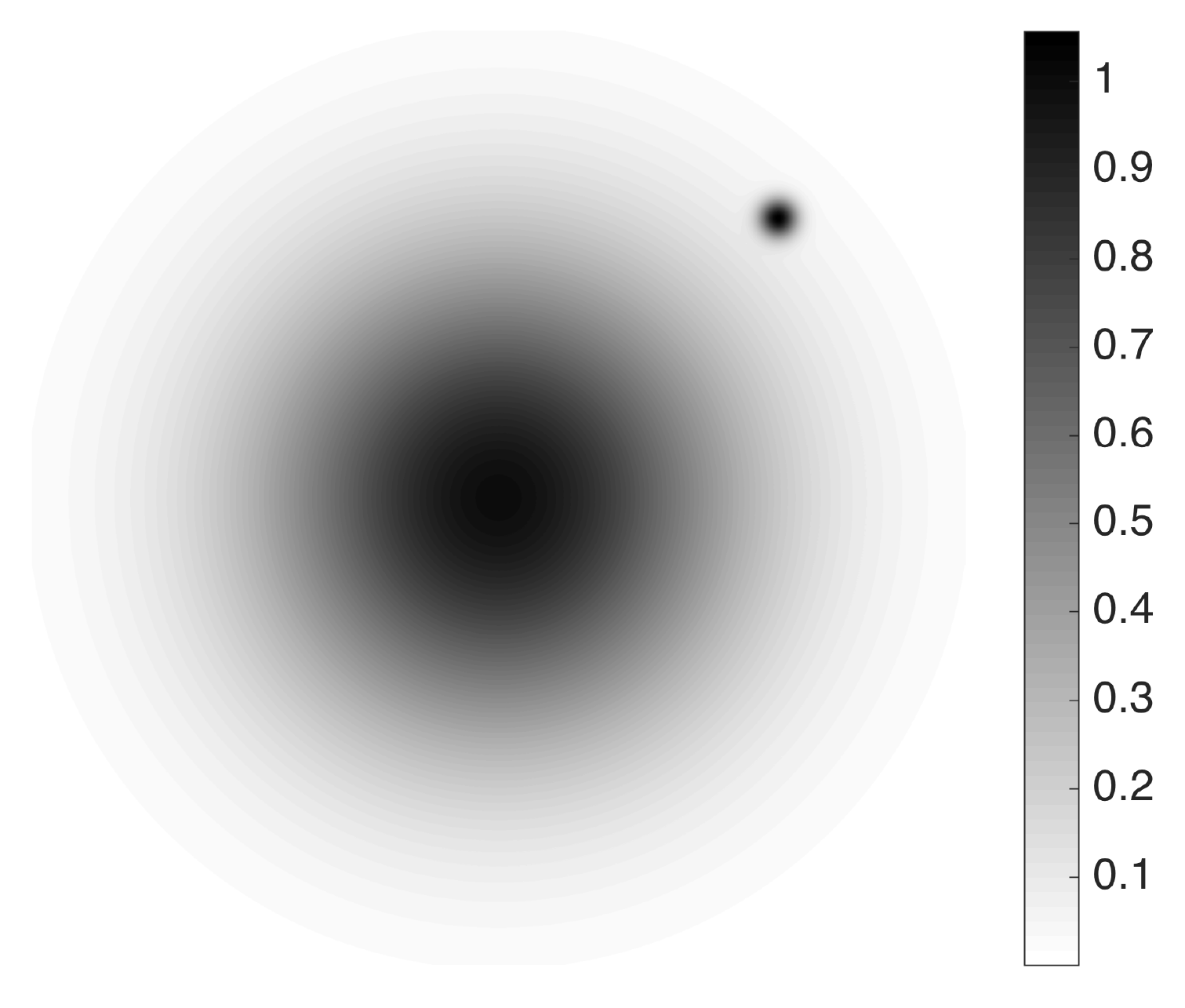}
\end{subfigure}
\caption{Left: Visualization of the boundary $\Gamma$ for Example 1.  For scaling tests with a fixed number of modified points, the size of the perturbation will vary depending on the size of $N$, while for those with a variable number of points the size of the perturbation will remain constant.  Right: Visualization of the perturbed scatterer function $w_1(x)$ for Example 2.  In our example, the size of the perturbing Gaussian will vary depending on the size of $N$. \label{fig:gauss}\vspace*{-.6cm}}
\end{figure}

We use the trapezoid rule to discretize \eqref{eq:bielap} on $\Gamma = \Gamma_1$, a circle with a bump function perturbation as in Figure \ref{fig:gauss}, whose radius is given by
\begin{align}
r(t) &= \left\{ \begin{array}{cl} 1 + 0.25\exp\left(\frac{-1}{1-[s(t)]^2}\right)& t\in (t_m,\,t_M),\\ 1&\text{else,} \end{array}\right.
\end{align}
with $s(t) = \frac{2t-(t_M + t_m)}{(t_M-t_m)}$,
and then factor the resulting system using $\rskelf{}$.  With this base factorization, we move the quadrature points where necessary such that they discretize $\Gamma_2,$ a simple circle (\ie, $r\equiv 1$), and change the necessary quadrature weights to reflect the new arc lengths.  We then use the old factorization as input to our updating algorithm to construct the factorization for the new geometry.

To investigate asymptotic scaling of the updating algorithm as we increase the number of discretization points $N$, there are two primary ways to increase the problem size.  The first is to choose $\Gamma_2$ to have a perturbed region that is independent of $N$, which implies that a fixed proportion of the discretization points will be modified.  The second is to use a variable-size perturbation region for $\Gamma_2$ such that the number of modified discretization points is constant.  In the first case, we expect to see linear scaling with $N$, since the number of modified leaf-level boxes is $\mathcal{O}(N)$, and in the second case theory dictates poly-logarithmic scaling.  For the first case, we will take $(t_m,\,t_M) = \left(\frac{9\pi}{10},\,\frac{11\pi}{10}\right),$ and for the second we take $(t_m,\,t_M) = \left(\pi-\frac{1000\pi}{N},\,\pi+\frac{1000\pi}{N}\right).$

The data for this example can be seen in Table \ref{tab:ex1}, with Figure \ref{fig:rskelresults} showing corresponding scaling results for both the case of updating a constant proportion of DOFs (approximately $N/10$) and a constant number of DOFs (approximately 1000).  The initial factorization time for both cases for fixed $N$ was approximately the same.

\begin{table}[h]
\caption{Timing results for Example 1 with \rskelf{} as we vary the ID tolerance $\epsilon$ and the total number of points $N$.  The time to construct the initial factorization is $t_\text{f}$.  We see that doubling the number of points doubles the time to update a constant proportion of points, $t_\text{u,p}$, but that the time to update in response to the modification of a constant number of points $t_\text{u,n}$ grows more slowly with $N$.\label{tab:ex1}}
\centering

\begin{tabular}{@{}ccccc@{}}
$\epsilon$ & $N$               & $t_{\text{f}} \,(s)$                       & $t_{\text{u,p}} \,(s)$                 & $t_{\text{u,n}} \,(s)$ \\ \midrule
\multirow{3}{*}{$10^{-3}$}     & 524288  &  9.4e$+$0         & 8.6e$-$1                    &  1.9e$-$2                   \\
                               & 1048576 &  1.9e$+$1           & 1.7e$+$0                    &  2.0e$-$2                   \\
                               & 2097152 &  3.8e$+$1           & 3.4e$+$0                    &  2.1e$-$2                   \\ \midrule
\multirow{3}{*}{$10^{-6}$}     & 524288  &  1.2e$+$1           & 1.1e$+$0                    &  3.0e$-$2                  \\
                               & 1048576 &  2.4e$+$1           & 2.1e$+$0                    &  3.1e$-$2                  \\
                               & 2097152 &  4.9e$+$1           & 4.2e$+$0                    &  3.2e$-$2                  \\ \midrule
\multirow{3}{*}{$10^{-9}$}     & 2097152 &  1.6e$+$1          & 1.5e$+$0                    &  5.0e$-$2                  \\
                               & 1048576 &  3.3e$+$1          & 2.8e$+$0                    &  5.3e$-$2                  \\
                               & 2097152 &  6.5e$+$1           & 5.6e$+$0                    &  5.6e$-$2 \\ \midrule
\end{tabular}
\end{table}

\begin{figure}
\centering
\begin{subfigure}{0.4\textwidth}
\centering
\hspace{-1.5cm}
\includegraphics[width=\textwidth]{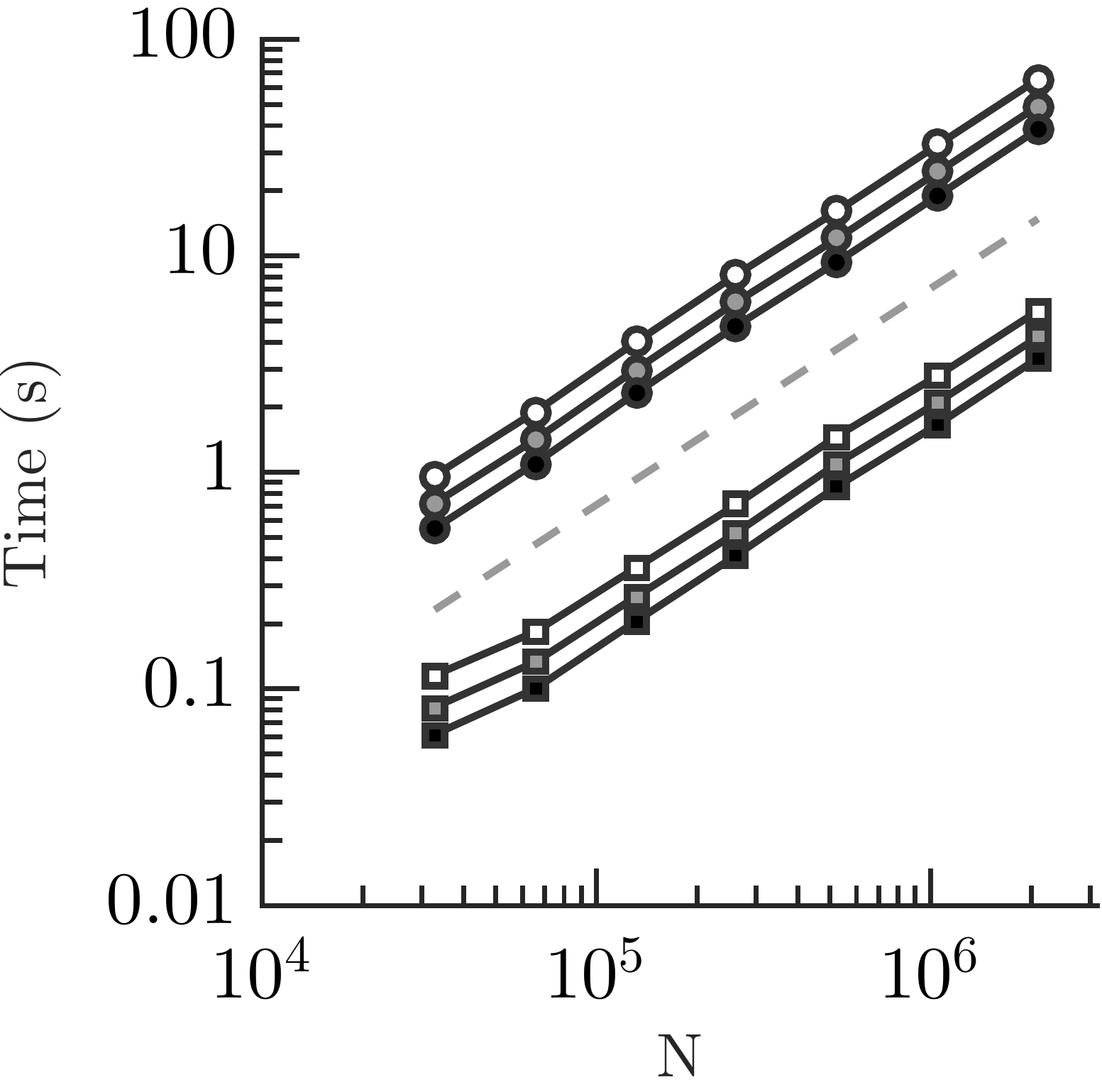}
\end{subfigure}
\quad
\begin{subfigure}{0.4\textwidth}
\centering
\includegraphics[width=\textwidth]{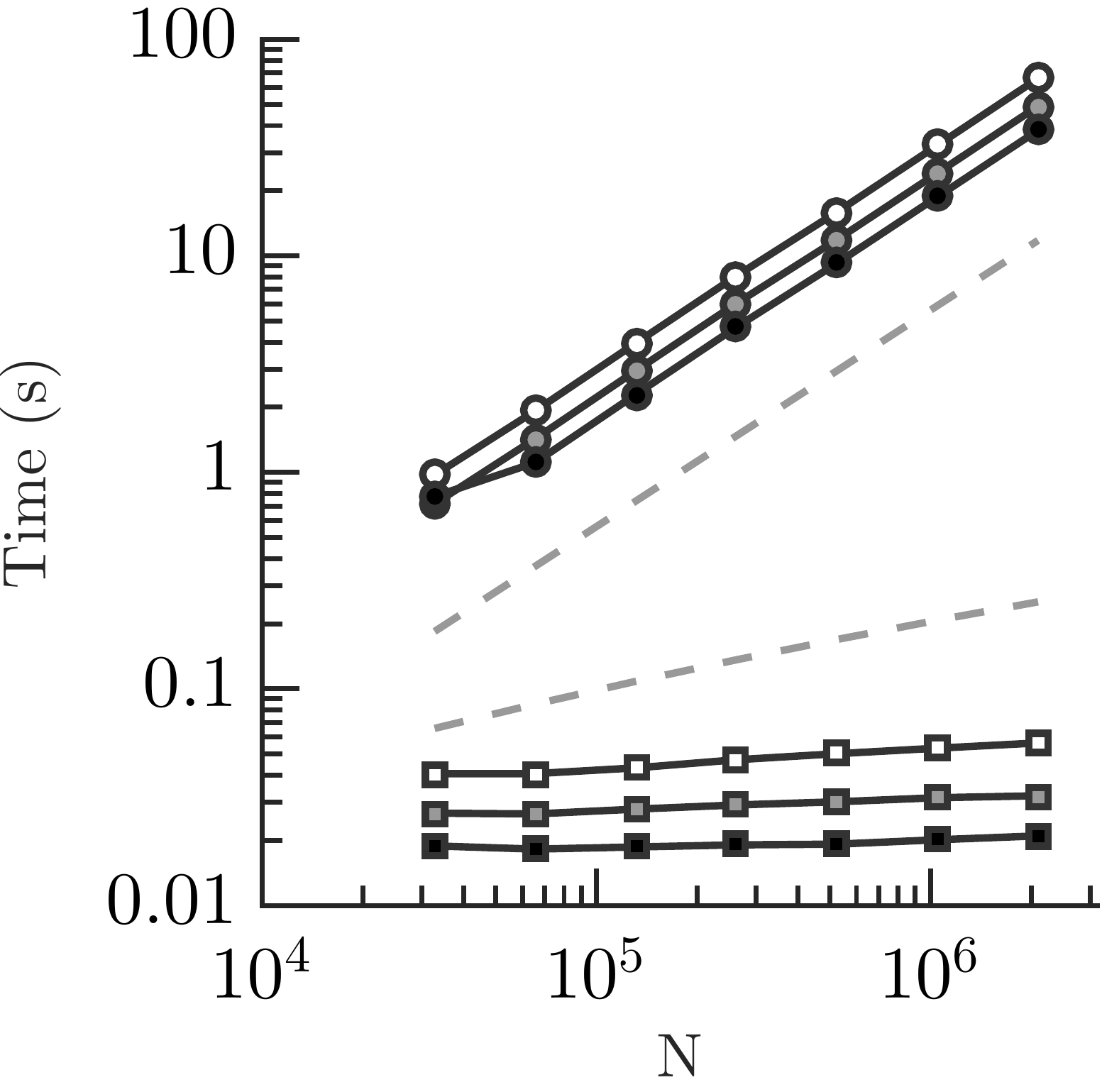}
\end{subfigure}
\caption{Timing results for Example 1 with \rskelf{} on the perturbed circle.  Circular markers denote factor times and square markers denote update times for tolerances $\epsilon$ of $10^{-3}$ (black), $10^{-6}$ (gray), and $10^{-9}$ (white).  Left: updating a fixed proportion of points with guide curve $\mathcal{O}(N)$.  Right: updating a constant number of points with guide curves $\mathcal{O}(N)$, and $\mathcal{O}\left(\log^4 \hspace{-0.06cm}N\right)$, from top to bottom. \label{fig:rskelresults} \vspace*{-.6cm}}
\end{figure}

\subsection{Example 2: the Lippmann-Schwinger equation}
To demonstrate updating of \hif{} for the true 2D case, we consider the Lippmann-Schwinger equation for Helmholtz scattering of an incoming wave with frequency $k$,
\begin{align}
\sigma(x) + k^2 \int_\Omega K(\|x-y\|)w(y)\sigma(y)\,d\Omega(y) &= f(x), &x\in\Omega=(0,1)^2.
\end{align}
Here, $K(r) = (i/4)H_0^{(1)}(kr)$ is the fundamendal solution of the Helmholtz equation written in terms of the zeroth order Hankel function of the first kind, $H_0^{(1)}(x)$, and $w(x)$ is a function representing the scatterer.  Although this kernel is derived from an elliptic partial differential equation, it is oscillatory with the frequency of oscillation and thus relative smoothness of the kernel dependent on $k$.
Assuming that $w(x)$ is non-negative, we can make the change of variables $u(x) = \sqrt{w(x)}\sigma(x)$ to obtain the symmetric form
\begin{align}\label{eq:ls}
u(x) +k\sqrt{w(x)}\int_\Omega K(\|x-y\|)\left[ k\sqrt{w(y)}\right]u(y)\,d\Omega(y) &= \sqrt{w(x)}f(x),\, x\in\Omega,
\end{align}
which affords a speedup of about a factor of two.

Given $w(x)$, we discretize \eqref{eq:ls} using a uniform $\sqrt{N}\times\sqrt{N}$ grid, where the diagonal entries $A_{ii}$ of the matrix $A$ are computed adaptively and the off-diagonal entries $A_{ij}$ are approximated using one-point quadratures.  For this example, we will consider starting with the function
\begin{align}
w_0(x) &= \exp(-16\|x-c\|^2),
\end{align}
a Gaussian centered at $c=\left[0.5,\,0.5\right]^T,$ and then modifying the scatterer by adding a perturbation that is essentially localized to construct
\begin{align}
w_1(x) &= w_0(x) + \exp(-s\|x-d\|^2),
\end{align}
where the perturbation is a Gaussian centered at $d=[0.8,0.8]^T$ truncated to machine precision and $s = s(N)$ is an adaptive scale parameter.  In particular, we choose $s(N)$ such that roughly 340 points lay within the region where the perturbation is greater than machine precision, which isolates the perturbation to a number of leaf-level boxes of the quadtree that is independent of $N$.  An example perturbed scatterer can be seen in Figure \ref{fig:gauss}.  The box in each test has sides of unit length, and we choose the frequency as $k=2\pi\kappa$ for wave numbers $\kappa = 0.1,$ $\kappa=1,$ and $\kappa=10$.  The data for this example can be seen in Table \ref{tab:ex2}, with the corresponding scaling plot in Figure \ref{fig:hifresults}.  We see that, for larger $\kappa$, correspondingly larger $N$ is required to reach the asymptotic regime.

\begin{figure}
\centering
\begin{subfigure}{0.3\textwidth}
\includegraphics[width=1.0\textwidth]{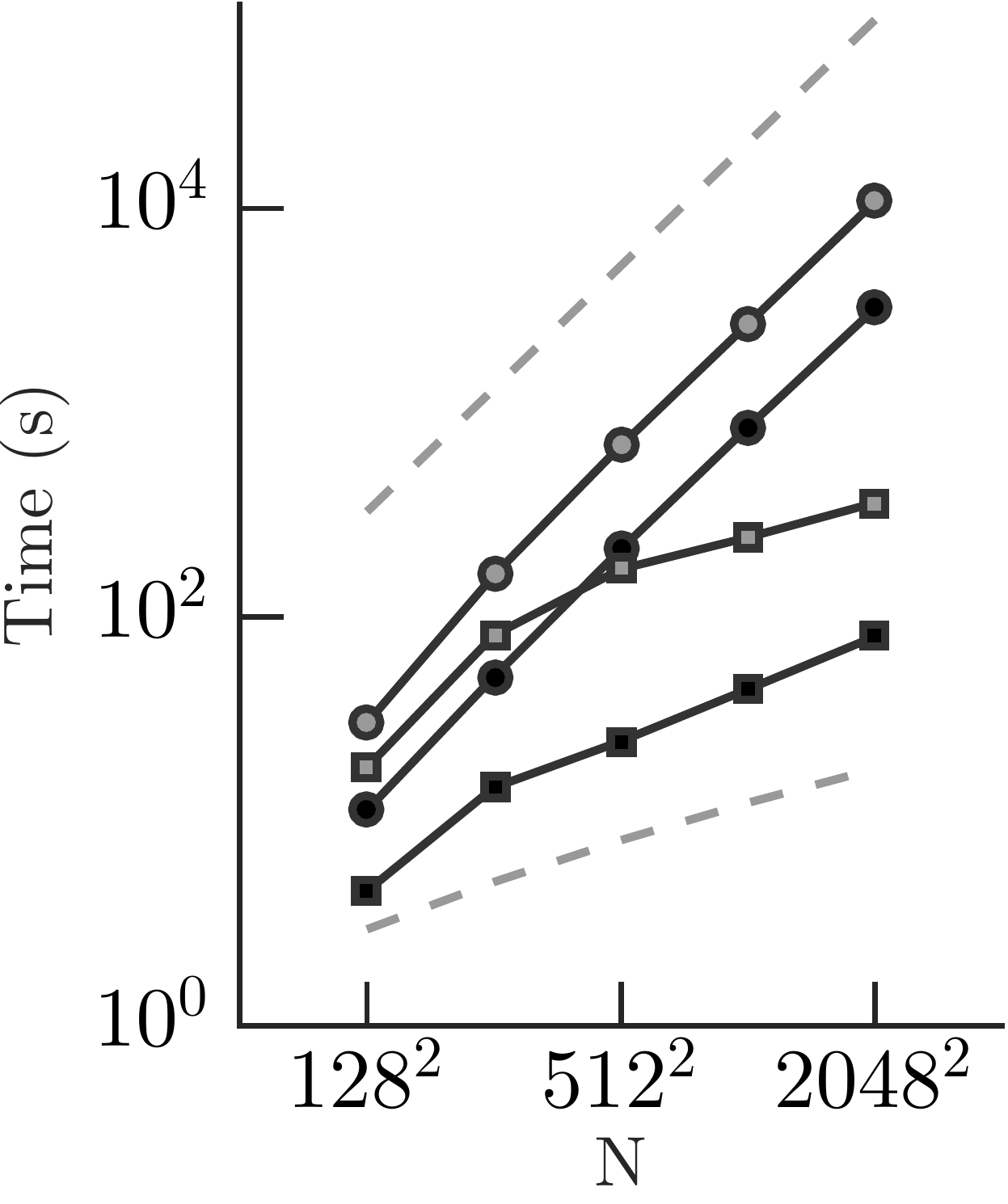}
\end{subfigure}
\;
\begin{subfigure}{0.3\textwidth}
\includegraphics[width=1.0\textwidth]{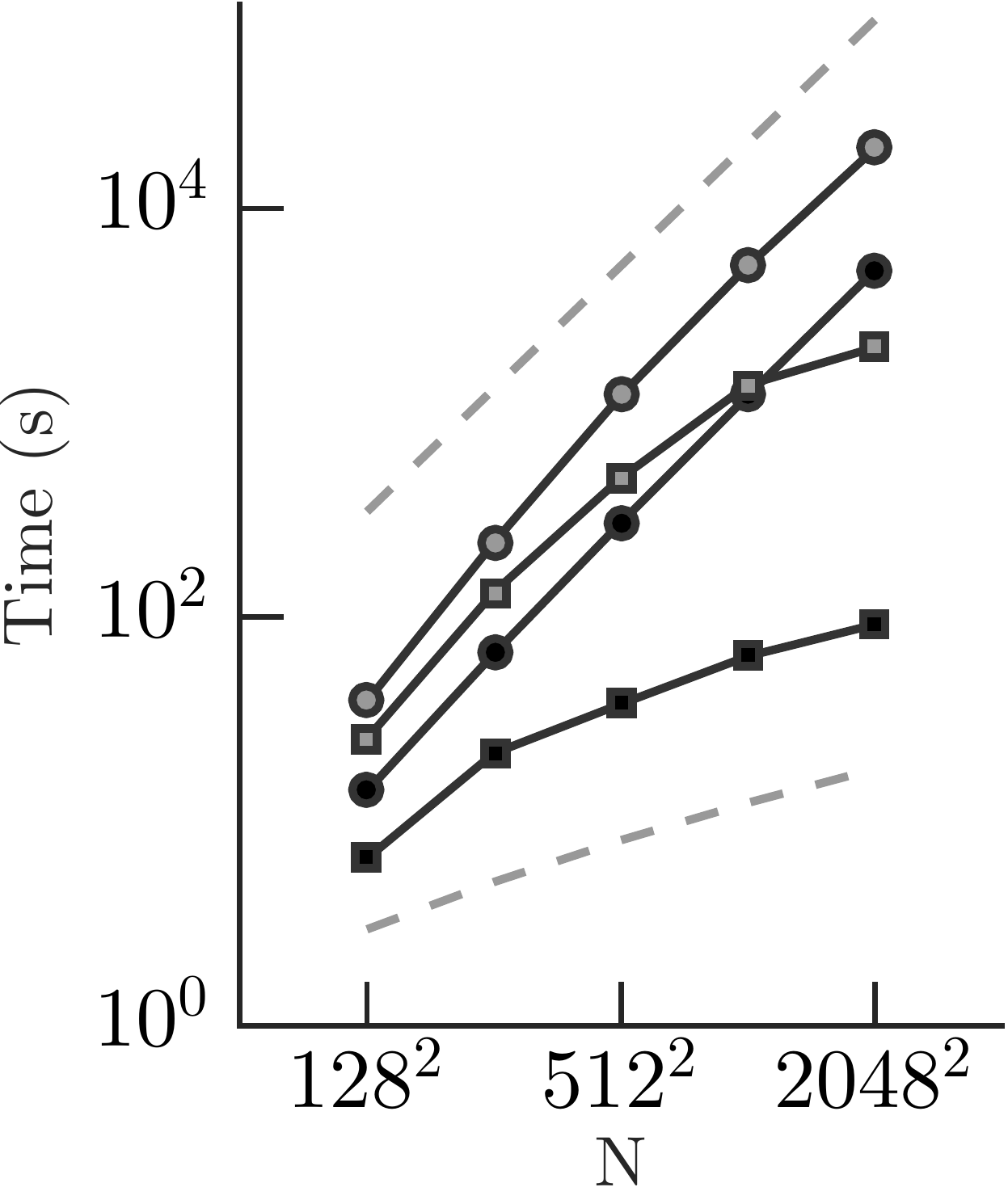}
\end{subfigure}
\;
\begin{subfigure}{0.3\textwidth}
\includegraphics[width=1.0\textwidth]{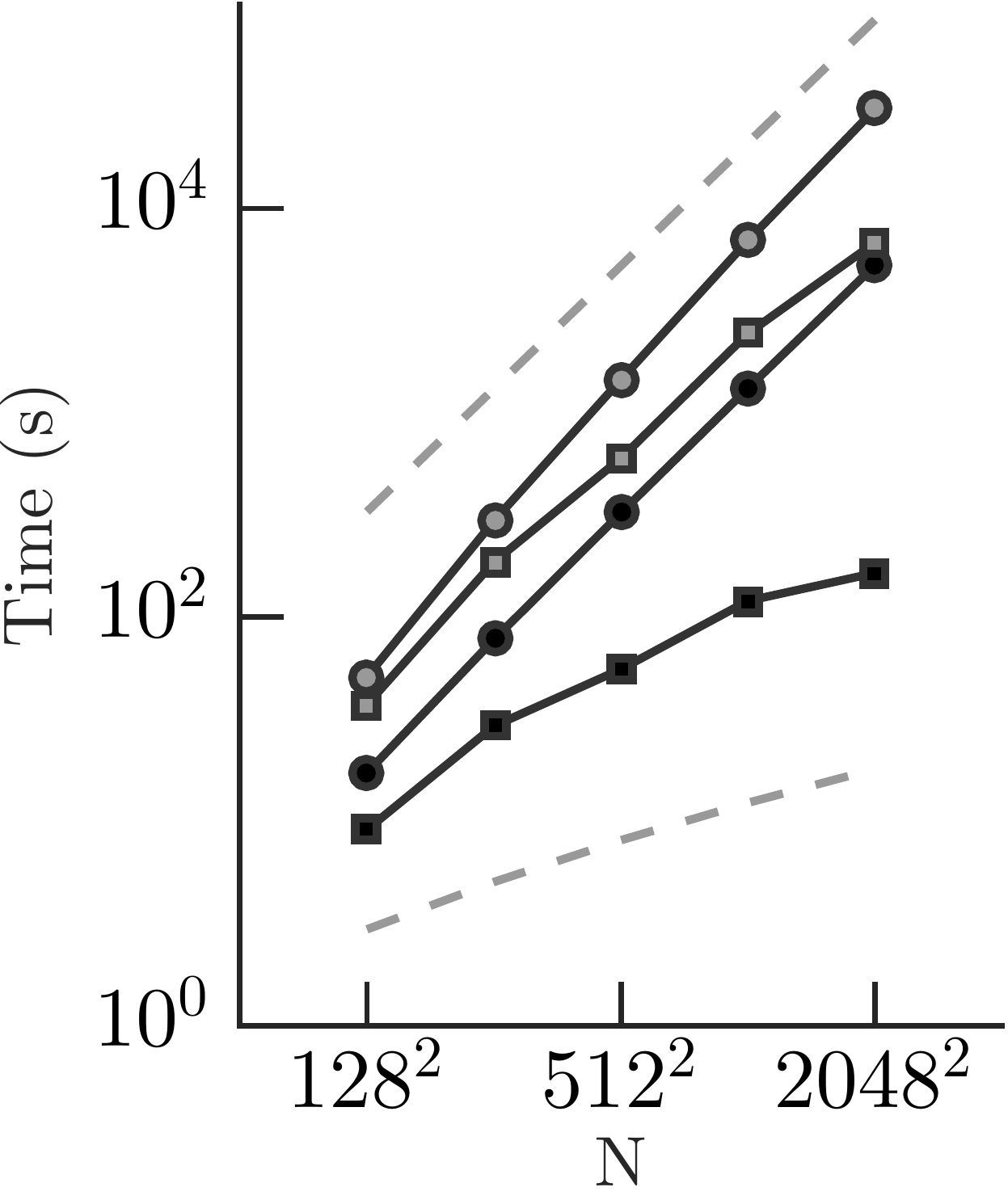}
\end{subfigure}
\caption{Timing results for Example 2 with \hif{} on the Lippmann-Schwinger example, updating a constant number of points for wave number $\kappa = 0.1$ (left), $\kappa=1$ (center), and $\kappa=10$ (right).  Circular markers denote factor times and square markers denote update times for tolerances $\epsilon$ of $10^{-3}$ (black) and $10^{-6}$ (gray).  The top guide line is $\mathcal{O}(N)$ and the bottom is $\mathcal{O}\left(\log^4 \hspace{-0.06cm}N\right)$.\label{fig:hifresults} \vspace*{-.6cm}}
\end{figure}

\begin{table}
\centering
\caption{Timing results for Example 2: Lippmann-Schwinger with \hif{} as we vary the ID tolerance $\epsilon$ and the total number of points $N$.  Note the slow growth with respect to $N$ of of the time to update in response to modifying a constant number of points, $t_{\text{u,n}}$, when the wave number $\kappa$ is small.\label{tab:ex2}}
\begin{tabular}{@{}cccccccc@{}}
&&\multicolumn{2}{c}{$\kappa=0.1$}&\multicolumn{2}{c}{$\kappa=1$}&\multicolumn{2}{c}{$\kappa=10$}
\\\cmidrule(lr){3-4}
\cmidrule(lr){5-6}
\cmidrule(lr){7-8}

$\epsilon$ & $N$               & $t_{\text{f}} \,(s)$                         & $t_{\text{u,n}}\, (s)$ & $t_{\text{f}} \,(s)$                         & $t_{\text{u,n}}\, (s)$ & $t_{\text{f}} \,(s)$                         & $t_{\text{u,n}}\, (s)$\\ \midrule
\multirow{3}{*}{$10^{-3}$}     & $512^2$  &  2.1e$+$2                            & 2.5e$+$1  & 2.9e$+$2& 3.8e$+$1  & 3.2e$+$2 & 5.6e$+$1                  \\
                               & $1024^2$ &  8.5e$+$2                            & 4.4e$+$1  & 1.2e$+$3& 6.4e$+$1  & 1.3e$+$3& 1.2e$+$2                \\
                               & $2048^2$ &  3.2e$+$3                            & 8.0e$+$1  & 5.0e$+$3& 9.2e$+$1   & 5.3e$+$3&1.6e$+$2               \\
                               \midrule
\multirow{3}{*}{$10^{-6}$}     & $512^2$  &  6.9e$+$2                            & 1.7e$+$2  & 1.2e$+$3& 4.8e$+$2   & 1.5e$+$3& 6.0e$+$2              \\
                               & $1024^2$ &  2.7e$+$3                            & 2.5e$+$2  & 5.2e$+$3& 1.4e$+$3   & 7.0e$+$3& 2.4e$+$3              \\
                               & $2048^2$ &  1.1e$+$4                            & 3.6e$+$2  & 2.0e$+$4& 2.1e$+$3   &3.1e$+$4& 6.8e$+$3              \\
                                \midrule

\end{tabular}
\end{table}

\begin{figure}
\centering

\end{figure}

\section{Conclusions}
\label{sec:conclusions}
Our examples indicate that the updating algorithm behaves as expected given our theoretical results, with linear scaling in the total number of leaf-boxes containing DOFs that have been directly modified and poly-log scaling in the total number of DOFs. This is a result that is perhaps not surprising theoretically, but should prove to be of great utility for real-world implementations of these algorithms.

In contrast to the Sherman-Morrison-Woodbury (SMW) strategy for solving perturbed systems in \cite{gg}, the end result of our algorithm is the factorization corresponding to the new system. Furthermore, this process is exact, \ie, it results in exactly the same approximate factorization as if a new one had been computed from scratch. One advantage of this strategy is that it allows for subsequent updates to a new set of localized DOFs, possibly located in a different region of the domain.

When a constant number $m$ of points are modified, our updating strategy has asymptotic cost $\mathcal{O}(m\log^4 N)$ to update and $\mathcal{O}(N)$ for the subsequent solve -- that is to say, updating is (asymptotically) essentially free if one is interested in using the updated factorization to solve a system.  Compared to the $\mathcal{O}(mN)$ cost of updating with the SMW strategy, we can obtain a significantly better asymptotic complexity considering that the small perturbations made in Example $2$ lead to $m$ on the order of 340.

For the recursive skeletonization factorization, our updating process is not difficult to implement.  Our examples show that in cases where we update even large portions of the domain it is possible to recover the constant factor complexity difference between updating and complete refactorization.  For the hierarchical interpolative factorization, we saw that updating requires more bookkeeping than for \rskelf{} due to the diagonal updates at the edge levels.  However, the updating algorithm still shows the same asymptotic scaling.

While here we have only discussed updating for 2D integral equations, all the ideas presented in this paper extend directly to the 3D case.  In fact, it is for 3D problems that we expect to see the biggest performance gain from using an updating procedure instead of completely refactoring the system.  The reason for this stems again from simple box counting -- the number of white nodes in the 3D analogue of Figure \ref{fig:hi} grows more quickly with respect to the depth of the tree.

Additionally, just as \hif{} can be carried out in the partial differential equation case as described in \cite{hifde}, so too can the updating procedure described here.  The ideas of box-marking and keeping track of diagonal interactions extend directly, but now the DOFs in the linear system come from, \eg, a finite difference discretization.  This is current work that we will present in a future publication.

\section*{Acknowledgments}
V.M. is supported by a U.S. Department of Energy Computational Science
Graduate Fellowship under grant number DE-FG02-97ER25308. A.D. is
partially supported by a National Science Foundation Graduate Research
Fellowship under grant number DGE-1147470 and a Simons Graduate
Research Assistantship.  K.H. is supported by a National Science
Foundation Mathematical Sciences Postdoctoral Research Fellowship
under grant number DMS-1203554.  L.Y. is partially supported by the
National Science Foundation under award DMS-1328230 and the
U.S. Department of Energy's Advanced Scientific Computing Research
program under award DE-FC02-13ER26134/DE-SC0009409.  The authors thank
L. Ryzhik for computing resources, as well as A. Benson,
B. Nelson, N. Skochdopole, and the anonymous reviewers for useful comments on drafts of this manuscript.

\bibliographystyle{siam}
\bibliography{M102450}
\label{LastPage}
\end{document}